\documentclass[8pt]{amsart}
\usepackage{amssymb,verbatim}
\usepackage[mathscr]{eucal}
\usepackage{enumitem}
\usepackage{amsthm,stackrel}
\usepackage{graphicx}
\usepackage{amscd,amsmath}
\usepackage{mathtools} 
\usepackage {epic}
\usepackage[all,color]{xy}
\usepackage[alphabetic]{amsrefs}
\usepackage{color}
\usepackage{hyperref}
\usepackage{adjustbox}
\hypersetup{
    colorlinks=true, 
    linktoc=all,     
    linkcolor=blue,  
}

\usepackage{amsfonts}
\usepackage{amsmath}

\newtheorem{prop}{Proposition}[section]

\newtheorem{coro}[prop]{Corollary}
\newtheorem{defi}[prop]{Definition}
\newtheorem{defi-lemm}[prop]{Definition-Lemma}

\newtheorem{exam}[prop]{Example}

\newtheorem{lemm}[prop]{Lemma}
\newtheorem{pf-thm}[prop]{proof of theorem}
\newtheorem{rema}[prop]{Remark}

\newtheorem{theo}[prop]{Theorem}
\newtheorem*{ack}{Acknowledgments}
\newcommand{\stkout}[1]{\ifmmode\text{\sout{\ensuremath{#1}}}\else\sout{#1}\fi}

\def\cD{\mathcal{D}}

\def\cJ{\mathcal{J}}
\def\cM{\mathcal{M}}

\def\cW{\mathcal{W}}

\def\cX{\mathcal{X}}

\def\sO{{\mathscr O}}

\def\sL{{\mathscr L}}

\def\sO{\mathscr{O}}

\def\sE{\mathscr{E}}

\def\CC{\mathbb{C}}

\def\GG{\mathbb{G}}
\def\HH{\mathbb{H}}
\def\JJ{\mathbb{J}}
\def\LL{\mathbb{L}}
\def\MM{\mathbb{M}}
\def\NN{\mathbb{N}}
\def\PP{\mathbb{P}}

\def\RR{\mathbb{R}}

\def\bD{\mathbf{D}}

\def\bJ{\mathbf{J}}
\def\bL{\mathbf{L}}
\def\bM{\mathbf{M}}

\def\faut{\mathfrak{aut}}

\def\fdiff{\mathfrak{diff}}

\def\fg{\mathfrak{g}}
\def\fgl{\mathfrak{gl}}

\def\fham{\mathfrak{ham}}
\def\fk{\mathfrak{k}}

\def\fp{\mathfrak{p}}

\def\ft{\mathfrak{t}}
\def\fu{\mathfrak{u}}
\def\fz{\mathfrak{z}}

\def\Ad{\mathrm{Ad}}
\def\ad{\mathrm{ad}}
\def\faut{\mathfrak{aut}}

\def\chow{\mathrm{Chow}}

\def\DF{\mathrm{Fut}}

\def\Diff{\mathrm{Diff}}
\def\Ding{\mathrm{Ding}}

\def\Div{\mathrm{div}}
\def\Dom{\mathrm{Dom}}
\def\End{\mathrm{End}}
\def\ev{\mathrm{ev}}

\def\GL{\mathrm{GL}}
\def\Ham{\mathrm{Ham}}
\def\ham{\mathrm{ham}}

\def\im{\mathrm{Im}}
\def\Int{\mathrm{int}}
\def\intg{\mathrm{int}}

\def\Lie{\mathrm{Lie}}
\def\map{\mathrm{Map}}

\def\dd{\mathrm{partial}}

\def\PSH{\mathrm{PSH}}
\def\Re{\mathrm{Re}}

\def\Ric{\mathrm{Ric}}

\def\SL{\mathrm{SL}}
\def\Solv{\mathrm{Solv}}
\def\Sp{\mathrm{Sp}}

\def\Sym{\mathrm{Sym}}
\def\SU{\mathrm{SU}}

\def\tr{\mathrm{Tr}}
\def\UU{\mathrm{U}}
\def\vol{\mathrm{vol}}

\def\ii{\sqrt{-1}}
\def\la{\langle}
\def\lla{\langle\!\langle}
\def\llp{(\!(}
\def\rrp{)\!)}
\def\rra{\rangle\!\rangle}
\def\ra{\rangle}
\def\ti{\widetilde}

\def\dbar{\bar \partial}
\def\dd{\partial}

\def\ii{\sqrt{-1}}

\newcommand{\red}{\color{black}}

\newcommand{\blue}{\color{black}}

\newcommand{\green}{\color{black}}

\def\al{\alpha}
\def\be{\beta}
\def\ep{\epsilon}

\def\de{\delta}
\def\De{\Delta}
\def\ga{\gamma}
\def\Ga{\Gamma}
\def\lam{\lambda}
\def\Lam{\Lambda}

\def\Om{\Omega}
\def\om{\omega}
\def\vphi{\varphi}

\def\si{\sigma}

\def\th{\theta}

\newcommand{\tri}{\triangle}

\newcommand{\ddbar}{\sqrt{-1}\partial\bar\partial}

\begin{document}

\title{Moment map, convex function and  extremal point}

\author{King-Leung Lee}
\address{Departamento de Matem\'{a}ticas\\
Facultad de Ciencias, Universidad Aut\'{o}noma de Madrid\\
28049 Madrid\\ Spain}
\email{king.lee@uam.es}


\author{Jacob Sturm}
\address{Department of Mathematics and Computer Science\\
           Rutgers University, Newark NJ 07102-1222\\ USA}
\email{sturm@rutgers.edu}

\author{Xiaowei Wang}
\address{Department of Mathematics and Computer Science\\
           Rutgers University, Newark NJ 07102-1222\\ USA}
\email{xiaowwan@rutgers.edu}

\date{\today}
\maketitle
\begin{abstract}
The moment map $\mu$ is a central concept in the study of Hamiltonian actions of compact Lie groups $K$ on symplectic manifolds. In this short note, we propose a theory of moment maps coupled with an $\Ad_K$-invariant convex function $f$ on $\fk^\ast$, the dual of Lie algebra of $K$, and study the properties of the critical point of $f\circ\mu$. Our motivation comes from Donaldson \cite{Donaldson2017} which is an example of infinite dimensional version of our setting. As an application, we interpret  K\"ahler-Ricci solitons as a special case of the generalized extremal metric.
\end{abstract}

\tableofcontents

\section{Introduction}
Let $(Z,\om,I)$ be a K\"ahler manifold equipped with K\"ahler form $\om$ and  complex structure $I$.
Suppose $(Z,\om)$
admits a Hamiltonian action of a {\red compact } Lie group $K$, that is, suppose there exists a {\red moment map}
\begin{equation*}
  \mu:Z\longrightarrow \fk^\ast=\mathrm{Lie}(K)^\ast \text{ satisfying }\
\left\{
\begin{array}{ccclc}
\la d\mu,\eta\ra_\Lie&=&\om(\cdot\, ,\si_z(\eta)) &{\rm for\ all}\  \eta\in \fk\\
\mu(k\cdot z)&=&\Ad^\ast_k\mu(z). &{\rm for\ all}\   k\in K
\end{array}
\right.
\end{equation*}
where  $\la\cdot\,,\,\cdot\ra_\Lie$ is the canonical pairing and $\si_z:\fk\to T_zZ$ denotes the infinitesimal action of $K$. We assume further that the $K$-action on $Z$ can be {\red complexified} to a $G$-action with $K\subseteq G$ and $\fg:=\Lie(G)=\fk\otimes_\RR\CC$, i.e. $G$ is the reductive group complexifying $K$. 
Let  $f$ be a smooth, strictly convex $\Ad^\ast_K$-invariant  $\mathbb{R}$ valued function on $\fk^\ast$.
Our first main result is the following generalization of Matsushima-Calabi decomposition (which is the case $f=|\mu|^2$):

\begin{theo}[ Theorem \ref{Calabi}]
 Suppose that $z$ is a critical point (necessarily a strict minimum) of $f\circ\mu$, i.e. suppose $df|_{\mu(z)}\in \fk_z$ where
$\fk_z$ is the stabilizer of $z$. Then we have decomposition:


$$
\fg_z=\bigoplus_{\lam \leq 0}\fg_\lam \text{ where } \fg_\lam:=\left\{\xi\in \fg\left|{}\, \ad_{\sqrt{-1}df|_{\mu(z)}}\xi\,=\, \lam \xi\right.\right\}.
$$
In particular, 
$\fg_0\subset \fg_z$ is the {\red reductive} part of $\fg_z$,
and
$df|_{\mu(z)}\in \fz\big(\fg_0\big)$. 
\end{theo}

Next we  define (cf. Theorem \ref{Fut}) the ``{\red $\mu$-invariant}" and prove that its non-vanishing is an obstruction to the existence of a critical point. When applied to the infinite dimensional setting of Fano manifolds in \cite{Donaldson2017}, the $\mu$-invariant is the well known Tian-Zhu \cite{TianZhu2002} generalization of  Futaki's invariant to the setting of K\"ahler-Ricci solitons.

To explain in more detail one of the motivations of this work, let $(X,\om_X)$ be a symplectic manifold and 
$$
\cJ_\Int(X,\om_X):=\{ J\in \Ga(\End(TX))\mid J^2=-I, \om_X(J\cdot, J\cdot)=\om_X(\cdot,\cdot),\ \om_X(\cdot,J\cdot)>0 \text{ and } N_{\om_X}(J)=0\}
$$
with $N_\om(J)$ being the Nijenhuis tensor associated to $J$. Then $\cJ_\Int(X,\om)$  is the space of integrable $\om_X$-compatible almost complex structures on $(X,\om_X)$, which is an infinite-dimensional K\"aher manifold
if equipped with the {\red Berdntsson} K\"ahler form $\lla \cdot,\cdot\rra$ defined in \cite[Theorem 1]{Donaldson2017}.\ Note, this is {\red different} from the K\"ahler form in \cite{Donaldson1997}. By applying the above Theorem  to the K\"ahler manifold
$(Z,\om):=(\cJ_\Int,\lla \cdot,\cdot\rra)$,  we obtain a generalization of Calabi's decomposition Theorem \cite[Theorem A]{TianZhu2000} first discovered by Tian and Zhu (revisited in the work of Nakamura \cite{Nakamura2018,Nakamura2019}). Furthermore,
Tian-Zhu's generalized Futaki invariant is a special case of the $\mu$-invariants we introduce in Theorem \ref{Fut}.

\begin{ack}
The work of the last author was partially supported by  a Collaboration Grants for Mathematicians from Simons Foundation:631318 and NSF:DMS-1609335.
\end{ack}

\section{Convex function on $\fk^\ast$}
Let $G$ be a reductive {\red linear } algebraic group,  $K\subseteq G$ be a fixed {\red maximal compact} subgroup, and
 $\frak{k}=\Lie(K)$.
For $\eta\in\fk$ and $\xi^*\in\fk^*$ we  write
$$
\eta(\xi^*):=\ \la\xi^*,\eta\ra\ = \ \xi^*(\eta) \text{ with } \la\cdot,\cdot\ra:\fk^\ast\times \fk\to \RR \text{ be the natural pairing.}
$$
Then if $k\in K$ we have
$$ \la\xi^*,\eta\ra\ = \la\Ad_k^*\xi^*,\Ad_k\eta\ra.
$$
 Let $f:\fk^*\rightarrow\RR$ be a {\red strictly convex} function which
is  $\Ad^\ast_K$-invariant, that is \begin{equation}\label{hess}
0<\la\cdot, \nabla^2f|_{\al^\ast}(\cdot)\ra:\Sym^2 \fk^\ast \longrightarrow \RR
\end{equation}
and
\begin{equation}\label{inv}
 f\circ \Ad_k^*\ = \ f\ \ \text{ for all}\ \ k\in K.
 \end{equation}
 \begin{defi}
Fix $\al^*\in\frak{k}^*$ then
$$
df_{\al^\ast}\in \fk.
$$
We define an isomorphism:
\begin{equation}\label{nabla-f}
\nabla^2 f\big|_{\al^\ast}:
\begin{array}{ccccc}
\fk^\ast & \xrightarrow{\hspace*{1.6cm}}  & \fk \\
\xi^\ast  & \xmapsto{\hspace*{1.0cm}}  & \xi:=\nabla^2 f\big|_{\al^\ast} (\xi^\ast)&
\end{array}
\end{equation}
The convexity of $f$ in \eqref{hess} allows us to define a positive definite {\red inner product} on $\fk^\ast$ for every $\al^\ast\in \fk^\ast$ as follows.
\begin{equation}\label{hess-norm}
\red
\llp\eta^*,\xi^*\rrp_{\al^*}:={\dd\over \dd s}{\dd\over \dd t}\bigg|_{_{s=t=0}}\, f(\al^*+t\xi^*+s\eta^*) =
\nabla^2f\big|_{\al^*}(\xi^*)(\eta^*)\ = \ \la\eta^*, \nabla^2f\big|_{\al^\ast}(\xi^*)\ra_{}\
\end{equation}
\end{defi}


\begin{prop}\label{prop1} For $k\in K$, and $\al^*,\ \be^*,\gamma^*\in \fk^*$ and $\xi,\ \eta\in \fk$  we have
\begin{itemize}
\item
\begin{equation}\label{1}
d f\big|_{\ga^*}(\be^*)\ = df\big|_{\Ad_k^*\ga^*}(\Ad_k^*\be^*)\in \RR\ \ 
\text{ or equivalently }\Ad_k (df|_{\ga^*})=d f|_{\Ad^*_k\ga^\ast}\in \fk.
\end{equation}

\item 
\begin{equation}\label{hess-Ad}
\big(\nabla^2f|_{\red\Ad^\ast_k\al^*} \big)\circ\Ad^\ast_k\be^*=\Ad_k\circ(\nabla^2 f|_{\al^\ast})(\be^\ast).
\end{equation}
\item
\begin{equation}\label{llp}
 \llp\al^*,\be^*\rrp_{\ga^*}\ = \ \llp\Ad_k^*\al^*,\Ad_k^*\be^*\rrp_{\red\Ad_k^*\ga^*}
\end{equation}
\item
\begin{equation}\label{33}
\red
\ad_\eta(df|_{\ga^\ast})  \ = \nabla^2 f\big|_{\ga^\ast}\big(  \ad^\ast_\eta(\ga^\ast)\big)\in\fk \text{ or equivalently }
\big(\nabla^2 f\big|_{\ga^\ast}\big)^{-1}\ad_\eta(df|_{\ga^\ast})  \ =  \ad^\ast_\eta(\ga^\ast)\in\fk^\ast
\end{equation}
 In particular, if we assume further that $\nabla^2 f$ is {\red invertible} then by letting $\ga^*=\mu(z), \eta=df|_{\mu(z)}$,  one obtains
\begin{equation}\label{grad-mu}
\ad^\ast_{df|_{\mu(z)}}\mu(z)=\big(\nabla^2 f\big|_{\red \mu(z)}\big)^{-1}\ad_{df|_{\mu(z)}}(df|_{\mu(z)})=0,\  \forall z\in Z.
\end{equation}

\item Suppose $\ad_\eta\ga^\ast=0$, then we have

\begin{equation}\label{de-0}
\nabla^2f\big|_{\ga^\ast}\circ \ad^*_\eta=\ \ad_\eta\circ \nabla^2f\big|_{\ga^\ast}:\fk^\ast \longrightarrow \fk, 
\text { hence } \nabla^2f\big|_{\mu}\circ \ad^*_{df|_\mu}=\ \ad_{df|_\mu}\circ \nabla^2f\big|_{\mu} \text{ by \eqref{grad-mu}}.
\end{equation}

\end{itemize}
\end{prop}

\begin{proof}
To prove \eqref{1}, we apply (\ref{inv}) to obtain

$$
d f\big|_{\ga^*}(\be^*)\ = \ {d\over dt}\bigg|_{t=0}\,f(\ga^*+t\cdot\be^*)\
= \ {d\over dt}\bigg|_{t=0}\,f(\Ad_k^*\ga^*+t\cdot\Ad_k^*\be^*)\ =\
d f\big|_{\Ad_k^*\ga^*}(\Ad_k^*\be^*).
$$
For \eqref{hess-Ad}, by applying \eqref{1} we obtain: for all $k\in K$
\begin{eqnarray*}
 &&( \nabla^2f|_{\Ad^\ast_k\al^\ast} )\cdot\Ad_k\be^\ast=
 \left.\frac{d}{dt}\right |_{t=0} df|_{\Ad^\ast_k(\al^\ast+t\be^\ast)} \stackrel{\eqref{1}}{=} \left.\frac{d}{dt}\right |_{t=0}\Ad_k(df|_{\al^\ast+t\be^\ast})
= \Ad_k\circ (\nabla^2 f|_{\al^\ast })(\be^*).
 \end{eqnarray*}
The proof of \eqref{llp} is similar, by definition \eqref{hess-norm} 
$$\llp\Ad_k^*\al^*,\Ad_k^*\be^*\rrp_{\Ad_k^*\ga^*}\ = \ {\dd^2\over\dd s\dd t}\bigg|_{(s,t)=(0,0))}
f(\Ad_k^*\ga^*+s\cdot\Ad_k^*\al^*+t\cdot\Ad_k^*\be^*)
$$
and now we apply (\ref{inv}) again to remove the $\Ad_k^*$ from the last expression.

To prove \eqref{33} we
first rewrite (1) as
$$ \la\xi^*,df\big|_{\al^*}\ra\ = \la\Ad_k^*\xi^*, d f\big|_{\Ad_k^*\al^*}\ra = \left\la \xi^*, \Ad_k^{-1} \big(df\big|_{\Ad_k^*\al^*}\big)\right\ra
$$
\begin{equation}\label{rewrite}
\Ad_k(df\big|_{\al^*})\ = df\big|_{\Ad_k^*\al^*}\in \fk
\end{equation}
Now we substitute $k=\exp(t\cdot \eta)$ so  $\Ad_k^*\ga^*= \ga^*+t\cdot\ad^*_\xi(\ga^*)+ O(t^2)$ and differentiate:
$$\ad^\ast_\eta(df\big|_{\ga^*})\ =\left.\frac{d}{dt}\right |_{t=0}\Ad^\ast_{e^{t\eta}}(d f|_{\ga^\ast})=\left.\frac{d}{dt}\right |_{t=0}(d f|_{\Ad^\ast_{e^{t\eta}}\ga^\ast})= \nabla^2f\big|_{\ga^*}\Big((\ad^*_\eta(\ga^*))\Big).
$$
To prove \eqref{de-0}, we substitute $k=\exp(t\eta)$ in \eqref{llp}. Then by our assumption $\ad^*_\eta(\ga^*)=0 \Longrightarrow \Ad_k^*\ga^*=\ga^*$ we have
$$ \llp\al^*,\be^*\rrp_{\ga^*}\ = \ \llp\Ad_k^*\al^*,\Ad_k^*\be^*\rrp_{\red\Ad_k^*\ga^*}=\ \llp\Ad_k^*\al^*,\Ad_k^*\be^*\rrp_{\red\ga^*}.
$$
 Now differentiating with respect to $t$ gives

\begin{eqnarray*}
 0
& =& \llp\ad^\ast_\eta(\al^*), \be^*\rrp_{\ga^*}\ +  \llp\al^*, \ad^*_\eta(\be^*)\rrp_{\ga^*}\\
& = &\   \la\ad^*_\eta(\al^*), \nabla^2f\big|_{\ga^\ast}(\be^*)\ra + \ \la\al^*, \nabla^2f\big|_{\ga^\ast}\big(\ad^*_\eta(\be^*)\big)\ra\\
& = & \ -\left\la\al^*, \ad_\eta\big(\nabla^2f\big|_{\ga^\ast}(\be^*)\big)\right\ra
+ \left \la\al^\ast, \nabla^2f\big|_{\ga^\ast}\big(\ad^*_\eta(\be^*)\big)\right\ra
\end{eqnarray*}
\end{proof}

\section{Moment maps and convex functions}\label{moment-convex}
Now let $(Z,\om,J)$ be a K\"ahler manifold with K\"ahler form $\om$ and complex structure $J\in \mathrm{End}(TZ)$.  The Riemannian metric is given as:
$$\la u,v\ra_{TZ}:=\om(u,Jv) \text{ and  }\la Ju,v\ra_{TZ}=\om(u,v).$$
Let $K$ be a compact Lie group acting  on $Z$ isometrically  with  $\mu:Z\to\fk^*$ being a $K$-equivariant moment map, that is,  for all $z\in Z$, $V\in T_zZ$ and $\eta\in \fk$ we have
\begin{equation}\label{kequi}
\mu(k\cdot z)\ = \Ad_k^*(\mu(z))
\end{equation}
\begin{equation}\label{moment}
 d\la \mu_z(v),\eta\ra\ = \ \la Jv,\si_z(\eta)\ra_{T_zZ}\ = \ \la Jv,\si_z(\eta)\ra_{_{T_zZ}}
\end{equation}

\begin{prop} \label{dmu}
Let $z\in Z$ and $\xi\in\fk$.  Then
\begin{equation}\label{dmu}
 d\mu_z(\si_z(\xi))\ =\ad^*_\xi(\mu(z))\ \stackrel{\eqref{33}}{=} \big( \nabla^2 f\big|_{\mu(z)}\big)^{-1}\big(\ad_\xi(df|_{\mu(z)}) \big)\in \fk^\ast
\end{equation}
In particular,
\begin{enumerate}
\item If $\de\in \fk_z$ then $\ad^*_\de\mu(z)=0$.
\item  if $\eta^*\in\fk^*$ then (\ref{moment}) with $v=\si(\xi)$ implies
\begin{equation}
\label{3}\llp\ad^*_{df|_{\mu(z)}}(\xi^*),\eta^*\rrp_{\mu(z)}\ = \ -\la {\red J}\si(\xi),\si(\eta)\ra_{T_z Z} \text{ with } \xi:=\nabla^2 f\big|_{\mu} (\xi^\ast),
\eta:=\nabla^2 f\big|_{\mu} (\eta^\ast).
\end{equation}

\end{enumerate}
In particular, $\ad^*_{df|_{\mu(z)}}: \fk^\ast\to \fk^\ast$ is {\red self-adjoint} with respect to the inner product $\llp\cdot,\cdot\rrp$ as it was so on the right hand side of the above equation.
\end{prop}
\begin{proof} For the first equality let $k=\exp(t\cdot\xi)$ and compute

$$\ad^*_\xi(\mu(z))= {d\over dt}\bigg|_{t=0} \Ad_k^*(\mu(z))\ = \
{d\over dt}\bigg|_{t=0} (\mu(k\cdot z))\ =d\mu_z(\si_z(\xi))
$$
The second identity is \eqref{33} in  Proposition \ref{prop1}. For the last identity, we have
Notice that we have
{\blue
\begin{eqnarray*}
&&\llp\,\ad^*_{df|_{\mu(z)}}(\xi^*),\eta^*\,\rrp_{\mu(z)}
=  \llp\,\eta^*,\ad^*_{df|_{\mu(z)}}(\xi^*)\,\rrp_{\mu(z)}=\left\langle\, \eta^*,(\nabla^2f|_\mu)[\ad^*_{df|_{\mu(z)}}(\xi^*)]\,\right\rangle\\
&\stackrel{\eqref{de-0}}{=} &
\langle  \,\eta^*,\ad_{df|_{\mu(z)}}(\nabla^2f|_\mu)(\xi^*)\,\rangle  {\red :\stackrel{\eqref{3}}{=}}
\langle\, \eta^*,\ad_{df|_{\mu(z)}}\xi\,\rangle\\
&\stackrel{\eqref{dmu}}{=}& \langle \,\eta^*,\nabla^2f|_\mu\circ d\mu_z(\sigma_z(\xi))\,\rangle  =
\llp\, \eta^*,d\mu_z(\sigma_z(\xi))\,\rrp_\mu=\llp\,d\mu_z(\sigma_z(\xi)), \eta^*\,\rrp_\mu
\\
&=&\left\langle\,d\mu_z(\sigma_z(\xi)), (\nabla^2f|_\mu)(\eta^*)\,\right\rangle
=\langle\, d\mu_z(\sigma_z(\xi)), \eta\,\rangle
= \om(\sigma_z(\xi),\sigma_z(\eta))\\
&=& -\la J\sigma_z(\xi),\sigma_z(\eta)\ra_{TzZ}
\end{eqnarray*}
}


\end{proof}
As a consequence of the set-up above,  {\red extremal } points can be obtained via minimizing the function $f(\mu(z))$ along $G\cdot z$. \begin{coro}
Let  $\red z(t):=g(t)\cdot z$ with $g(t)\in G$ satisfying:
\begin{equation}\label{dg}
\fg\ni\frac{dg(t)}{dt}g^{-1}(t)=\ii\cdot df|_{\mu(z(t))}\in\ii\fk.
\end{equation}
If we write $g(t)=e^{t\ii\xi} g$ with $\xi\in \fk$ then
$$
\dfrac{dg(t)}{dt}g^{-1}(t)=\ii \xi \ {\red \overset{z(t)=g(t)\cdot z}{ \xRightarrow{\hspace*{1.5cm}}} }\
I\circ\si_{z(0)}(\xi)=\left.\frac{d z(t)}{dt}\right|_{t=0}={\red -}\si_{z(0)}(d f|_{\mu(z(0))}).
$$
That is,
the {\red negative gradient} flow of $f\circ \mu$.

Moreover, we have
$$
\frac{d}{dt}\left({\red -}\log\frac{| g(t)\cdot \hat z|}{|\hat z|}\right)=-\left\la\mu(z(t)),\frac{1}{\ii}\frac{dg(t)}{dt}g^{-1}\right\ra=-\la\mu(z(t)), df|_{\mu(z(t))}\ra
$$
and if we assume further that $0<\la\cdot,\nabla^2 f\big|_{\al^\ast}(\cdot)\ra: \Sym^2\fk^\ast\longrightarrow \RR$ for all $\al^\ast\in \fk^\ast$ and $f(0)=0$ then we will have
$$
\frac{d}{dt}\left({\red -}\log\frac{| g(t)\cdot \hat z|}{|\hat z|}\right)=-\la\mu(z(t)), df|_{\mu(z(t))}\ra\leq - f\circ \mu(z(t)).
\footnote{
For $f\in C^2(\fk^\ast), \ \xi^\ast\in \fk^\ast$ with $\nabla^2 f\geq 0$, we have
\begin{eqnarray*}
f(\xi^\ast)-f(0)
&=&\int_0^1 \frac{d}{dt}(f(t\xi^\ast))dt=\int_0^1 \la \xi^\ast,  df|_{\red t\xi^\ast}\ra dt\\
\left (\because \frac{d}{dt}\la \xi^\ast, df|_{\red t\xi^\ast}\ra_\fk=\la\xi^\ast,\stackbin[\red\in \fk]{}{\underbrace{ \nabla^2 f|_{\red t\xi^\ast}(\xi^\ast)}}\ra_\fk\geq0\right )
&\leq &\la \xi^\ast, df|_{\red 1\cdot \xi^\ast}\ra
\end{eqnarray*}
from which we deduce
$$
f(\mu(z))=f(\mu(z))-f(0)\leq \la \mu(z),df|_{\red \mu(z)}\ra
$$
}
$$
which already appeared in \cite{Donaldson2017}.

\end{coro}

\subsection{Calabi-Matsushima decomposition}
In this section, prove a generalized version of classical Calabi-Matsushima decomposition \cite{Calabi}.
Let $\fg=\fk\otimes_\RR\CC$ and
$$
\fg_z=\{\xi_1+\ii\xi_2\in \frak{g}\,:\, \si(\xi_1)+J\si(\xi_2)=0\}\subset \fg.
$$
\begin{theo}\label{Calabi}
Assume that $z\in Z$ is a critical point of $f\circ\mu$, which will be called an {\red $f$-extremal point}.  Then
\begin{enumerate}
\item $d f\big|_{\mu(z)}\in \fk_z=\{\xi\in \fk\,:\, \si(\xi)=0\}$.
\vskip .03in
\item $\ad^*_{df|_{\mu(z)}}\mu(z)=0$.
\vskip .03in
\item We have
$$ \fg_z\ = \ \bigoplus_{\lam\leq 0} \,\frak{g}_\lam\ \ {\rm where}\ \
\fg_\lam\ = \ \{\xi\in \fg_z\,:\, \ii\ad_{df|_{\mu(z)}}\xi=\lam\,\xi\,\}
$$
\item In particular, we have
$ \fg_0\ = \ \fk_z\otimes_\RR\CC\cong \fg_{\rm red}:=\fg/\fg_{\rm solv}
$
where $\fg_{\rm solv}\subseteq \fg$ is the maximal solvable subalgebra.
\end{enumerate}
\end{theo}
\begin{proof}
To prove (1) we assume $z$ is a critical point of $f\circ\mu$ so $d(f\circ\mu)_z(v)=0$ for all tangent vectors $v\in T_zZ$. Thus
$$df|_{\mu(z)}(d\mu_z(v))=\la d\mu_z(v), df|_{\mu(z)}\ra=0 \text{ for all }v\in T_zZ.$$
Now \eqref{moment} implies $\la Iv,\si_z(df_{\mu(z)}\ra=0$ for all $v$, hence $\si_z(df\big|_{\mu(z)})=0$.

Part (2) follows from Part (1) of Proposition \ref{dmu}.

To prove (3), we let $k=\exp(t\cdot df|_{\mu(z)})$. Then (1) implies $k\cdot z=z$ so $\mu(z)=\Ad_k^*\mu(z)$ by (\ref{kequi}). Thus
if we substitute $k$ and $\al^\ast=\mu(z)$ into part (2) of Proposition \ref{prop1}, we have

$$
\llp\xi^*,\eta^*\rrp_{\mu(z)}\ = \ \llp\Ad_k^*\xi^*,\Ad_k^*\eta^*\rrp_{\mu(z)}
$$
and differentiate with respect to $t$ we get

\begin{equation}\label{anti}
\llp\ad^*_{df|_{\mu(z)}}\xi^*,\eta^*\rrp_{\mu(z)}\ +  \llp\xi^*,\ad^*_{df|_{\mu(z)}}\eta^*\rrp_{\mu(z)}\ = \ 0
\end{equation}
or equivalently, combining Proposition \ref{prop1} part (4) with the identification:
$$
\nabla^2 f\big|_{\al^\ast}:
\begin{array}{ccccc}
\fk^\ast & \xrightarrow{\hspace*{1.6cm}}  & \fk \\
\xi^\ast  & \xmapsto{\hspace*{1.0cm}}  & \xi:=\nabla^2 f\big|_{\al^\ast} (\xi^\ast)&
\end{array}
$$
we obtain

\begin{equation}\label{anti1}
\llp\ad_{df|_{\mu(z)}}\xi,\eta\rrp_{\mu(z)}\ + \ \llp\xi,\ad_{df|_{\mu(z)})}\eta\rrp_{\mu(z)}\ = \ 0 \text{ for all } \xi,\eta\in\fk.
\end{equation}
Thus $\lla\cdot,\cdot\rra_{\mu(z)}$ extends uniquely to a {\red Hermitian} inner product
on
$\fg:=\fk\otimes\CC$ for which the linear operator
$$\ii\ad_{df|_{\mu(z)}}:\fg\longrightarrow \fg \text{ is Hermitian}.
$$
In particular, all of its eigenvalues are {\red real} numbers.
Moreover, $df|_{\mu(z)}\in \fk_z$ so $\ii df_{\mu(z)}\in \fg_z$.
To prove (3), we need the following:
\begin{lemm} Let $z$ be a critical point of $f\circ\mu$ and assume
$$\ii\ad_{df|_{\mu(z)}}\xi=\lam\xi$$
for some $\xi\in\fg_z$
 Then $\lam\leq 0$.
\end{lemm}

\begin{proof}  Write $\xi=\xi_1+\ii\xi_2$ with $\xi_1,\xi_2\in \fk$. Then

$$\lam\xi_1+\lam\ii \xi_2=\lam\xi\ = \ \ii\ad_{df|_{\mu(z)}}(\xi_1+\ii \xi_2)\ \Longrightarrow \
\left \{
\begin{array}{lcc}
 \ad_{df|_{\mu(z)}}\xi_1&=&\lam\xi_2\\
  & & \\
 \ad_{df|_{\mu(z)}}\xi_2&=&-\lam\xi_1.
\end{array}
\right.
$$
Now (\ref{3}) implies

$$\lam\llp\xi_2,\xi_2\rrp_{\nabla^2f|_{\mu(z)}}=\llp\ad_{df|_{\mu(z)}}(\xi_1),\xi_2\rrp_{\nabla^2f|_{\mu(z)}}\ \stackrel{\red \eqref{3}}{=} \ -\la J\si(\xi_1),\si(\xi_2)\ra_{TZ}\ = \ -\la\si(\xi_2),\si(\xi_2)\ra_{TZ}
\ = \ -|\si(\xi_2)|^2_{TZ}
$$
where in the last equality the fact that $\xi\in \fg_z$ so $\si(\xi_1)+J\si(\xi_2)=0$ is used. Thus
$\lam\leq 0$ with equality $|\si(\xi_2)|_{TZ}=|\si(\xi_1)|_{TZ}=0$, i.e. if and only if $\xi\in\fk_z\otimes\CC\subset\fg_z$.
\end{proof}
Part (3) and (4) follows from the Lemma above, our proof is thus completed.
\end{proof}
\begin{exam}
Notice that the convexity of $f$ is {\red necessary} in Theorem \ref{Calabi} above. Let us consider $K\times K$-action on $Z\times Z$ with 
$$
f:
\begin{array}{ccccc}
\fk^\ast\oplus \fk^\ast & \xrightarrow{\hspace*{1.6cm}}  & \RR & &\\
(\al^\ast,\be^\ast)  & \xmapsto{\hspace*{1.0cm}}  & |\al^\ast |^{2}-|\be^\ast |^{2} &
\end{array}
$$
Let $z_0\in Z$ be a {\red extremal} point, i.e. $\mu_Z(z_0)\in \fk_{z_0}$. 
Then $(z_0,z_0)\in Z\times Z$ is  a $f$-extremal point  as 
$$df|_{\mu_{Z\times Z}(z_0,z_0)}=(\mu_Z(z_0),-\mu_Z(z_0))\in \fk_{z_0}\oplus \fk_{z_0}.$$
Then 
$$
\begin{array}{ccccc}
\ii\ad_{df|_{\mu_{Z\times Z}(z_0,z_0)}}(\xi,0)&=&[\ii (\mu_Z(z_0),-\mu_Z(z_0)), (\xi,0)]&=&\lam\cdot(\xi,0)\\
\ii\ad_{df|_{\mu_{Z\times Z}(z_0,z_0)}}(0,\xi)&=&[\ii (\mu_Z(z_0),-\mu_Z(z_0)), (0,\xi)]&=&-\lam\cdot (0,\xi)
\end{array}
$$
that is 
$$\ii\ad_{df|_{\mu_{Z\times Z}(z_0,z_0)}}:\fg_{z_0}\oplus \fg_{z_0}\longrightarrow \fg_{z_0}\oplus \fg_{z_0}$$ 
is {\red indefinite} contrast to Theorem \ref{Calabi} where $f$ is assumed to be {\red convex}.
\end{exam}

\subsection{$\mu$-invariant and extremal vector fields}
In this subsection, we introduce an {\red obstruction}  for the existence of an $f$-extremal point (defined in Theorem \ref{Calabi}) and introduce the {\red extremal vector field} for those $z\in Z$ when the $\blue\dim_\RR \fk_{z}$ is {\red maximal} (cf. Definition \ref{ex-def}).

\begin{lemm}\label{R-C}
Consider the $G$-action on the polynomial ring $\CC[\fg^\ast]$ induced by the $\Ad^\ast_G$-action on $\fg^\ast$, we have $\CC[\fg^\ast]^G\cong\CC[\ft^\ast]^W$ with $W$ being the {\red Weyl group} of $G$. In particular,
\begin{enumerate}
\item any $f\in \RR[\ft^\ast]^W=\RR[\fk^\ast]^K$ is a restriction of a unique $f\in \CC[\fg^\ast]^{G}$. (cf. Reference can be found \href{https://mathoverflow.net/questions/182871/generators-of-invariant-polynomials-of-semisimple-lie-algebra}{here}.)
\item
$$
df|_{(\bullet)}:
\begin{array}{ccccc}
\fg^\ast & \xrightarrow{\hspace*{1.6cm}}  & \fg & &\\
\xi^\ast  & \xmapsto{\hspace*{1.0cm}}  & df|_{\xi^\ast}&\in &T^\ast_{\xi^\ast} \fg^\ast\equiv \fg
\end{array}
$$
satisfying
\begin{equation}\label{Ad-G-f}
d f\big|_{\Ad_g^*\al^*}\ = \Ad_g (df\big|_{\al^*})\in \fg,\ \forall g\in G.
\end{equation}
Moreover, for $\xi^\ast_0\ne\xi^\ast_1\in\fk^\ast$ satisfying   $\nabla^2 f|_{\overline{\xi^\ast_0\xi^\ast_1}}\geq 0$ in the sense of \eqref{hess} with
$\overline{\xi^\ast_0\xi^\ast_1}\subset \fk^\ast$ denoting the line joining $\xi^\ast_0$ and $\xi^\ast_1$.  Then
$$
\la\xi^\ast_1-\xi^\ast_0, df|_{\xi^\ast_1}\ra-\la \xi^\ast_1-\xi^\ast_0,df|_{\xi^\ast_0}\ra>0
$$
with $\red =$ if and only if $\nabla^2 f\Big|_{\overline{\xi^\ast_0\xi^\ast_1}}\equiv 0$. In particular, if one assume that
\begin{equation}\label{conv}
\nabla^2 f|_{\fk^\ast}\geq 0\text{ and }
\mathrm{graph}(f):=\{(\xi^\ast, f(\xi^\ast))\mid \xi^\ast\in \fk^\ast\}\in \fk^\ast\times \RR
\text{ contains {\red no line segment.} }
\end{equation}
Then $df|_{(\bullet)}:\Ad^\ast_g\fk^\ast\longrightarrow \Ad_g \fk$ is one-to-one for any $g\in G$.
\end{enumerate}

\end{lemm}
\begin{proof}
By our assumption that $G$ is reductive, we conclude  part  (1) from the Harish-Chandra isomorphism and Proposition \ref{prop1}-(1).
In particular, we have isomorphism of the GIT quotients  $\fg^\ast/\!/G\cong \ft^\ast/\!/W$.

Next we prove the last part. The proof of identity \eqref{Ad-G-f} is the same as the part (1) of Proposition \ref{prop1}. Let $\xi^\ast_0\ne\xi^\ast_1\in\fk^\ast$ such that $df|_{\xi^\ast_0}=df|_{\xi^\ast_1}$ and
$\xi^\ast_t=(1-t)\xi^\ast_0+t\xi^\ast_1\in\fk^\ast$ then we have
\begin{eqnarray*}
\la\xi^\ast_1-\xi^\ast_0, df|_{\xi^\ast_1}\ra-\la \xi^\ast_1-\xi^\ast_0,df|_{\xi^\ast_0}\ra
&=&\int_0^1\frac{d}{dt}\la\xi^\ast_1-\xi^\ast_0, df|_{\xi^\ast_t}\ra dt\\
&=&\int_0^1\la\xi^\ast_1-\xi^\ast_0,\nabla^2 f|_{\xi^\ast_t}\cdot (\xi^\ast_1-\xi^\ast_0)\ra dt>0.
\end{eqnarray*}
Hence, our claim follows from \eqref{Ad-G-f}.
\end{proof}

\begin{exam}
Let $\fg=\fgl(n)$ and $\fg>\fk=\fu(n)$. Then  $f(\xi)=-\tr(\xi^k),\ \xi\in \fg$ extends $f(\xi)=|\xi|^{k/2}, \forall \xi\in \fu=\{\xi\in\fgl |\xi=-\bar\xi^t\}$. And
$df|_\xi=k \xi^{k-1}\in \fgl(n)$.
\end{exam}

We recall the definition of the moment map:

\begin{equation}\label{def}
 d\la\mu_z(v),\xi\ra\ = \om(v,\si_z(\xi))= \la Jv,\si_z(\xi)\ra_{_{T_zZ}}\ \ \text{for all $\xi\in\fk$ and $v\in T_zZ$.}
\end{equation}
In the presence of a complex structure, by comparing the $(0,1)-, (1,0)-$ components
$$
(\bullet)^{(1,0)}:={1\over 2}(1-\ii J)(\bullet), \ \ \ (\bullet)^{(0,1)}:={1\over 2}(1+\ii J)(\bullet)\ \ {\rm and} \ \
\dbar f(v)\ := \ {1\over 2}(df(v)+\ii df(J\cdot v)),
$$
 we  obtain for all $\xi\in\fk$ and $v\in T_z^\CC Z$.
\begin{eqnarray}\label{def1}
2\left\la\dbar\mu\big|_z(v),\xi\right\ra
&=&\left\la (d-\ii J\circ d)\mu\big|_z(v),\xi\right\ra=\left\la d\mu\big|_z(v),\xi\right\ra+\left\la \ii d\mu\big|_z(J\cdot v),\xi\right\ra\\
&=&\om \big(v^{(0,1)},\si^{}_z(\xi)\big)=  \om\Big(v,\big(\si_z(\xi)\big)^{\red(1,0)}\Big)\nonumber
\end{eqnarray}
Now we claim that \eqref{def1} holds for all $\xi\in \fk\otimes_\RR\CC=\fg$,
 { but this follows from the fact that the  {\red holomorphic $G$-action} on $Z$ naturally extends the $\RR$ linear morphism $\si_z:\fk\longrightarrow T_zZ$  to a {\red $\CC$-linear} morphism $\si_z:\fg\longrightarrow T^\CC_zZ$.}
\footnote{
More explicitly,
it suffices to prove it holds if $\xi\in\fk$ is replaced by $\ii\xi$.
For this, we multiply both sides of (\ref{def})
by $\ii $. This has the effect of replacing $\xi$ by $\ii\xi$ on the left. On the right, we get
$$ \ii\cdot\big(\si(\xi)\big)^{1,0}\ = \ \ii\cdot {1\over 2}(1-\ii J)\si(\xi) \ = \ {1\over 2}(1-\ii J)J\si(\xi)\ = \ {1\over 2}(1-\ii J)\si(\ii\xi)
\ = \ \big(\si(\ii\xi)\big)^{1,0}
$$
i.e. $\si$ is a $\CC$-linear map.
}
This establishes (\ref{def1}).
\begin{lemm}\label{holo-char}
Let $C\subset Z$ be any complex subvariety and $\{\xi(z)\}:C\to \fg$ being holomorphic function satisfying $\xi(z)\in \fg_z$.
Then $\la\mu(z),\xi(z)\ra$ is a holomorphic in $z\in C$.
\end{lemm}
\begin{proof}
It follows from \eqref{def1} and the assumption that $\dbar \xi=0$ and $\xi(z)\in\fg_z$ that
$$\dbar\la\mu,\xi\ra\big|_z=\stackbin[\eqref{def1}]{}{\underbrace{\la \dbar \mu,\xi\ra|_z}}+\stackbin[\dbar \xi=0]{}{\underbrace{\la\mu,\dbar \xi\ra|_z}}=0.$$
\end{proof}

\begin{theo} [\bf $\mu$-invariant]\label{Fut}
Let $\red \xi\in \fk_z, \ \eta\in \fg_z$ and $\phi: \fg\longrightarrow  \RR$ satisfying:
\begin{equation}\label{Ad-xi}
\phi(\Ad_g \xi)=\phi(\xi) \text{ for all }g\in G.
\end{equation}
Then
\begin{enumerate}
\item $\la d\phi|_{\Ad_g\xi},\Ad_g\eta \ra$ is independent of $g\in G$. 
\item The function $\chi: G\to \CC$ defined via
\begin{eqnarray}\label{chi}
\chi(g):= \la \mu(g\cdot z)- d\phi |_{\Ad_g\xi},\Ad_g\eta \ra
=\la \mu(g\cdot z),\Ad_g\eta \ra-\la d\phi |_{\Ad_g\xi},\Ad_g\eta \ra
\end{eqnarray}
is constant.
\end{enumerate}
\end{theo}
\begin{proof}
The first part follows from our assumption $\phi(\Ad_g \xi)=\phi(\xi)$ and the proof of  Proposition \ref{prop1} (1).

For the second part, the independence of $g\in G$ for the {\red first} term in \eqref{chi} was proved in \cite[Proposition 6]{Wang2004}. We give more conceptual proof as follows.  By the $\Ad_K$-equivariance of $\mu$, one obtains that $\chi |_{K\cdot z}$ is constant function. On the other hand, Lemma \ref{holo-char} implies that $\chi|_{G\cdot z}$ is holomorphic as $\Ad_g(\xi)$ varies holomorphically along $G\cdot z\subset Z$. So $\chi|_{G\cdot z}$ must be a constant function since $\chi |_{K\cdot z}$ is and $K\cdot z\subset G\cdot z$ is a totally real subvariety. 
%

For the second term in \eqref{chi}, the independence of $g\in G$ follows from part (1) of Lemma \ref{R-C} and  the fact
\begin{equation}\label{f-Ad-g}
\la d\phi (\Ad_g\xi), \Ad_g\eta\ra
\stackrel{\eqref{Ad-G-f} }
{=}\la\Ad_g^\ast\big( d\phi (\xi)\big),\Ad_g\eta\ra
=\la d\phi (\xi),\eta\ra.\\
\end{equation}
\end{proof}
\begin{rema}\label{no-convex}
Notice that  {\red no convexity } is assumed for $\phi$. This is particular the case in \cite{HanLi2020}'s work on $g$-soliton, that is {\red well-defindness} of Futaki-invariants needs no convexity.
\end{rema}
\begin{exam}\label{trace}
Let $\chi:\fg\to \CC$ be Lie algebra morphism 
\footnote{In particular, it descends to a linear mophism $\fg/[\fg,\fg]\to \CC$.}
and $\psi\in \CC[\fg]$ (or even more generally, $\psi$ is an analytic function defined over $\fg$) with $(\CC[\fg],[\bullet,\bullet])$ being equipped with a natural {\red polynomial Poisson algebra} structure induced from $(\fg,[\bullet,\bullet])$. Then $\phi:=\chi\circ\psi$ satisfies the assumption \eqref{Ad-xi} as
$$
\left.\frac{d}{dt}\right|\phi\big(\Ad_{e^{t\eta}}\xi\big)=\chi\left(\psi'(\xi)\cdot [\eta,\xi]\right)=\chi\left([\psi(\xi),\xi]\right)=0.
$$
\end{exam}

\begin{defi}
Let $f:\fk^\ast\to \RR$ be a {\red convex } function and for any $\xi\in \fk$ let
$$
\displaystyle \widehat f(\xi ):=\sup _{\eta^\ast \in \fk^\ast}(\eta^{*}(\xi)-f(\eta^\ast)):\fk\longrightarrow \RR,\ \ \ \ \xi\in\fk
$$
be the {\red Legendre transform } of $f$,  and $\forall \eta\in \fg$ we write $\red (\eta f)(\xi):=(\eta f)\big|_\xi$ by regarding $\eta\in \fg=T_\xi\fg$ as a vector field on $\fg$.
\end{defi}
\begin{coro}\label{chevalley}
Let $f\in \RR[\ft^\ast]^W=\RR[\fk^\ast]^{\Ad^\ast_K}$ satisfying $\nabla^2 f|_{\fk^\ast}>0$ in the sense of \eqref{hess}. Then
for $\xi\in \fk_z$ and $\eta\in \fg_z$ the function
\begin{eqnarray}\label{Fg'}
F(g):&=& \la \mu(g\cdot z)-d\widehat f|_{\Ad_g\xi},\Ad_g\eta \ra
=\la \mu(g\cdot z),\Ad_g\eta \ra-\la d\widehat f|_{\Ad_g\xi},\Ad_g\eta \ra 
\end{eqnarray}
is independent of $g\in G$,
where $d\widehat f=(df)^{-1}$ is defined via:
$\fg^\ast\ni d\widehat f|_\xi=(df)^{-1}(\xi):=\eta^\ast\Longleftrightarrow df|_{\eta^\ast}=\xi\in \fg.$
%
\end{coro}
\begin{proof}
By Chevalley isomorphism, we know that every $f\in \RR[\ft^\ast]^W=\RR[\fk^\ast]^{\Ad^\ast_K}$ is the restriction of a function which by abusing of notation still denoted by $f\in \CC[\fg^\ast]^{\Ad^\ast_G}$. Hence assumption \eqref{Ad-xi} is met and our claim follows from Theorem \ref{Fut}.
\end{proof}
\begin{rema}\label{gen-Chevalley}
Indeed, the Chevalley type isomorphism holds in more general sense, which was included in Appendix \ref{ext}.
\end{rema}
\begin{defi}
We define  the {\red $\mu$-invariant} for $(z,\xi)\in Z\times \fk_z$ as:
$$
\la \mu(z)-d\widehat f|_\xi,\eta\ra=\la \mu(z),\eta\ra-{\blue (\eta \widehat{f})|_\xi}
\text{ for } \eta\in \fg_z.
$$
In particular, if $z\in Z$ is a {\red critical } point of $f\circ\mu$ then
$$
df|_{\mu(z)}=\xi\in \fk_z \ {\red \xLeftrightarrow{\hspace*{1cm}}} d\widehat f|_\xi=\mu(z) \text{ and hence }\la \mu(z)-d\widehat f|_\xi,\eta\ra=0, \ \ \ \forall  \eta\in \fk_z.
$$
\end{defi}


\begin{exam}
Let $Z=\PP^n$ and $\mu(z)=\ii\dfrac{zz^\ast}{|z|^2}$ then for any $A\in \GL(n+1)$ and $\xi\in \fu(n+1)_z$ then  we have
$\xi\cdot z=\lam \cdot z$ and
\begin{eqnarray*}
\la \mu_\PP(A\cdot z), \Ad_A \xi\ra
&=&\tr \left(\frac{Azz^\ast A^\ast}{|Az|^2}\cdot A\xi A^{-1}\right)=\frac{\tr (Azz^\ast A^\ast\cdot A\xi A^{-1})}{|Az|^2} \\
&=&\frac{\tr (zz^\ast A^\ast\cdot A\xi )}{|Az|^2}=\frac{\tr (z^\ast A^\ast\cdot A\cdot \xi z )}{|Az|^2}\\
\big(\because\xi\cdot z=\lam z\big)
&=&\lam\cdot \frac{\tr (z^\ast A^\ast\cdot A z )}{|Az|^2}=\lam
\end{eqnarray*}
\end{exam}

\begin{lemm}\label{uni-ex}
For any $z\in Z$, there is {\red at most one}  $\xi_f\in \fk_z$ when exists solving
\end{lemm}
\begin{equation}\label{ex-vf-0}
\la \mu(z)-d\widehat f|_{\xi_f},\eta\ra_\fk=0,\ \forall\eta\in \fk_z
\end{equation}
\begin{proof}
Let $\xi_i\in\fk_z,\ i=0,1$ and $\xi_t=(1-t)\xi_0+t\xi_1\in\fk_z$. Then we have
\begin{eqnarray*}
&&\la (df)^{-1}(\xi_1),\xi_1-\xi_0\ra-\la(d f)^{-1}(\xi_0),\xi_1-\xi_0\ra\\
&=&\int_0^1\frac{d}{dt}\Big\la(df)^{-1}(\xi_t),\xi_1-\xi_0\Big\ra \cdot dt\\
\end{eqnarray*}
Differentiating
$df|_{df^{-1}(\xi_t)}=\xi_t$ implies:
$$
\nabla^2f|_{df^{-1}(\xi_t)}\Big(\stackbin[\red\in\fk^\ast]{}{\underbrace{\dfrac{d}{dt} \big(df^{-1}(\xi_t)\big)}}\Big)=\frac{d\xi_t}{dt}=\xi_1-\xi_0
$$
from which we deduce
\begin{eqnarray*}
&&\int_0^1\frac{d}{dt}\Big\la(df)^{-1}(\xi_t),\xi_1-\xi_0\Big\ra \cdot dt\\
&=&\int_0^1\la(\nabla^2 f|_{(d f)^{-1}(\xi_t)})^{-1} \cdot (\xi_1-\xi_0),\xi_1-\xi_0\Big\ra\cdot dt\\
(\eqref{nabla-f})
&>&0
\end{eqnarray*}
as $(\nabla^2 f)^{-1}>0$ by our assumption \eqref{conv},
from which we deduce that extremal vector $\xi_z$ must be  {\red unique}  by letting
$
\eta=\xi_0-\xi_1\in \fk_z.
$

The final conclusion follows from our assumption by letting $\xi=\xi_f\in \fk_z$ be the {\red unique } solution to
$$(df)^{-1}(\xi)-\pi_{\fk_z^\ast} (\mu(z)) \in \fk_z^{\perp}:=\{\xi^\ast\in \fk_z^\ast\mid \la \xi^\ast,\eta\ra=0,\ \forall \eta\in \fk_z\}\subset \fk^\ast_z.$$
\end{proof}
\begin{defi}\label{ex-def}
Let $z\in Z$ and $f$ satisfies the convexity assumption \eqref{conv} and 
\begin{itemize}
\item
\begin{equation}\label{solv}
\blue\dim_\RR\fk\cap \fg_{z}=\dim_\RR \fk_z=\dim_\CC (\fg_z/(\fg_z)_{\rm Solv}\big)\ .\ 
\footnote{This condition is necessary thanks to Theorem \ref{Calabi}, where $(\fg_z)_\Solv\subset\fg_z$ is the maximal solvable subalgebra.}
\end{equation}
that is, $\dim\fk_z$ achieves maximal dimension among $z'\in G\cdot z$, and
\item  The composition of morphism
$$
\fz(\fk_z)  \overset{d\widehat f=(df)^{-1}}{ \xrightarrow{\hspace*{1.5cm}}}  \fk ^\ast \overset{\pi_{\fk^\ast_z}}{ \xrightarrow{\hspace*{1.5cm}}} \fz(\fk_z)^\ast_z
$$
is {\red surjective},\ where 
$\fz(\fk_z)\subset\fk_z$ is the center of $\fk_z$ and
$$\pi_{\fk^\ast_z}:\fk^\ast\longrightarrow \fz(\fk_z)^\ast_z\  \text{ is the projection induced by the inclusion } \fz(\fk_z)\subset \fk.$$
\end{itemize}

Let   $\xi_f\in \fz(\fk_z)$ be the  {\red unique} solution (whose
existence is guaranteed by 
Lemma \ref{uni-ex}) 

\begin{equation}\label{ex-vf}
\la \mu(z)-d\widehat f|_{\xi_f},\eta\ra_\fk=0,\ \forall\eta\in \fk_z.
\end{equation}
We call $\xi_f$ the \emph{\red $f$-extremal vector field } at $z$ inspired by the work of  \cite{FutakiMabuchi1995}. 
In particular, $z$ is {\red $f$-extremal} when $df|_{\mu(z)}=\xi_f$.
\end{defi}
\begin{rema}
Notice that the \eqref{ex-vf} condition implies that $\xi_f\in \fz(\fk_z)\subset\fk_z$, the center of $\fk_z$. To see that, one notice that for all $\eta\in \fk_z$ and  $k\in K_z$, then 
\[
0=\langle \Ad_{k}\mu(z)-\Ad_{k}(df|_{\xi}), \Ad_{k}\eta\rangle
=\langle \mu(z)-df|_{\Ad_{k}{\xi}}, \Ad_{k}\eta\rangle
\] 
Since $\Ad_k:\fk_z\to\fk_z$ is an isomorphism, we obtain $\Ad_k\xi=\xi$ by the uniquenss of {\red extremal vector field} (cf. Lemma \ref{uni-ex}), hence $\xi_f\in \fz(\fk_z)$ as we claimed.
\end{rema}

Our next result, stating that if $\xi_f(z)$ is the $f$-extremal vector field at $z$ then for any $z'\in G\cdot z$ satisfying the condition \eqref{solv}, there is a unique $\xi_f(z')$ canonically related to $\xi_f(z)$. More precisely, we have the following:

\begin{coro} Let $z\in Z$ be a point satisfying \eqref{solv} and  $\xi_f(z)\in \fk_z$ is the {\red  $f$-extremal vector field} at $z$.
For any $z'=g\cdot z\in Z,\ g\in G$  at which \eqref{solv} is satisfied, there is a $g'\in G_{g\cdot z}$ such that 
$$\xi_f(z')=\xi_f(g\cdot z)=\Ad_{g'\circ g}\xi_f(z)\in \fk_{g\cdot z}=(\Ad_g\fg_z)\cap \fk$$
is the {\red $f$-extremal vector field} at $z'=g\cdot z\in Z$.
\footnote{Although $g,g'$ are not necessarily unique,  $\xi_f(z')$ is {\red uniquely} determined by $z'\in Z$. One should also notice that $g'$ seems to be necessary as $\Ad_g\xi_f(z)$ might not be in $\fk$.}
\end{coro}
\begin{proof}
By our assumption \eqref{solv}
$$\dim_\CC(\fk_z:=\fg_z\cap \fk)=\dim_\CC\Big(\fk_{g\cdot z}:=\fg_{g\cdot z}\cap \fk=(\Ad_g \fg_z)\cap \fk\Big)=\dim_\CC \fg_z/(\fg_z)_\Solv, $$
 both $\Ad_g(\fk_z=\fg_z\cap \fk),\ \fg_{g\cdot z}\cap\fk\subset\fg_{g\cdot z}$ are the Lie algebras of maximal compact subgroups of $G_z$. In particular, they are conjugate to each other inside $\fg_z$ by the Cartan-Iwasawa-Malcev Theorem, hence 
\begin{equation}\label{k-g'g}
\exists \ g'\in G_{g\cdot z}, \text{ such that } \Ad_{g'}\circ\Ad_g(\fk_z)=\fk_{g\cdot z}\subset\fg_{g\cdot z}
\end{equation}
since $\fk_{g\cdot z}=\Ad_g \fg_z\cap \fk=\fg_{g\cdot z}\cap \fk$. By Theorem \ref{Fut} and our assumption $\xi_f(z)$ being the {\red $f$-extremal vector field} at $z$, $\forall\eta\in \fk_z$ we have
\begin{eqnarray*}
0=\left\la \mu(z)-d\widehat f|_{\xi_f(z)},\eta\right\ra_\fg
&=&\left\la \mu\big((g'\circ g)\cdot z\big)-d\widehat f|_{\Ad_{g'\circ g}\xi_f(z)},\Ad_{g'\circ g}\eta\right\ra_\fg\\
\Big(\because \ g'\in G_{g\cdot z}\Big)
&=&\left\la \mu\big(g\cdot z\big)-d\widehat f|_{\Ad_{g'\circ g}\xi_f(z)},\Ad_{g'\circ g}\eta\right\ra_\fg
\end{eqnarray*}
as $\Ad_{g'\circ g}\xi_f(z),\Ad_{g'\circ g}\eta\in \fk_{g\cdot z}=\fk_{z'}$ by \eqref{k-g'g}. By the uniqueness $\xi_f$ in Definition \ref{ex-def}, one obtains
$$\xi_f(z')=\xi_f(g\cdot z)=\Ad_{g'\circ g}\xi_f(z)$$
which depends only on $g\cdot z\in Z$ but not on $g,g'$.

\end{proof}

\begin{exam}
Let $\Big(Z=(\PP^1)^{d}, \om_Z=\sum_{i=1}^d\pi_i^\ast \om_{\PP^1}\Big)$ and $K=\SU(2)<\SL(2)$ and let $z_t(d'):=d'\{[1,t]\}+(d-d')\{[0,1]\}\in Z$, then
for $d'>0$ we have
$$
\fk_{z_t(d')}=
\left\{
\begin{array}{ccc}
\fu(1)=\ii\RR=\RR\cdot\xi_0 & & t=0\\
 & & \\
 0 & & t\ne 0
\end{array}
\right.
\text{ and }\ \
\red
\xi_{f}=
\left\{
\begin{array}{ccccc}
\xi_0 & & t=0, & \dim \fk_{z_t}=1\\
 & & & &\\
\nexists & & t\ne 0, & \dim \fk_{z_t}=0
\end{array}
\right.
$$
even though $z_t(d/2)$ are  polystable.
\end{exam}

\begin{exam}
Let $\la\bullet,\bullet\ra_\fk$ denote the Killing form on $\fk$, which induces an isomorphism of $\fk\to \fk^\ast$, hence equips $\fk^\ast$ with an inner product $|\bullet|_{\fk^\ast}$. If
$f=|\bullet |^2_{\fk^\ast}$ then the extremal vector field is the unique vector $\xi_f\in \fk_z$ satisfying:
$$
\la \xi_f, \eta\ra_\fk=\la\mu(z),\eta\ra_\Lie \ \text{ for all }\eta\in \fk_z.
$$
Notice that $f(\xi^\ast)=|\xi^\ast|^2_{\fk^\ast}=-\tr((\xi^\ast)^2)\in \RR[\fk^\ast]^W$, thanks to the fact
$$
\fk:=\{\xi\in \fg \mid \xi=-\theta(\xi)\}\subset \fg .
$$
with $\theta:\fg\to\fg$ being the Cartan involution for $\fk$. Hence, $|\xi^\ast|^2_{\fk^\ast}\in \RR[\fk^\ast]^W$ is the restriction of $\tr((\xi^\ast)^2)\in\CC[\fk^\ast]^{\Ad_G}$ to $\fk$.
\end{exam}

\section{Applications in K\"ahler geometry}
Now we apply the theory we developed in the previous section to K\"ahler geometry.
Let $(X,\om_X)$ be a symplectic manifold with a {\red pre-quantum} complex line bundle $(L,h,\nabla)\to X$ equipped with a Hermitian metric $h$ together with a $h$-compatible connection $\nabla$ such that its curvature form satisfies: $\Ric(\nabla)=\om_X$.
Let
$$
\cJ_\Int(X,\om_X):=\{ J\in \Ga(\End(TX))\mid J^2=-I, \om_X(J\cdot, J\cdot)=\om_X(\cdot,\cdot),\ \om_X(\cdot,J\cdot)>0 \text{ and } N_{\om_X}(J)=0\}
$$
with $N_\om(J)$ denoting the Nijenhuis tensor associated to $J$.
So  $\cJ_\Int(X,\om)$  is the space of integrable $\om_X$-compatible almost complex structures on $(X,\om_X)$, which is an infinite-dimensional K\"aher manifold
if equipped with two K\"ahler forms
\begin{itemize}
\item Donaldson-Fujiki's K\"ahler form $(\cJ_\Int,\Omega_\cJ)$ introduced in \cite{Fujiki1990} and \cite{Donaldson1997}.
\item the {\red Berdntsson's} K\"ahler form $(\cJ_\Int,\lla \cdot,\cdot\rra_\cJ)$ defined in \cite[Theorem 1]{Donaldson2017}, which is different from the one defined in \cite{Donaldson1997}.
\end{itemize}
Let $\Ham(X,\om)$ be the group of Hamiltonian diffeomorphism, thanks to the work of \cite{Donaldson1997} and \cite{Donaldson2017}, both cases admit a moment map. We will apply the theory developed in the previous sections to the K\"ahler manifold
$(Z,\om):=(\cJ_\Int,\lla \cdot,\cdot\rra) 
.$


\subsection{ $\Ham(X,\om_X)$-action on $\cJ_\Int(X,\om_X)$}
\begin{defi}\label{Th}
A vector field $\xi\in \Ga(TX)$ is {\red Hamiltonian} with respect to $\om_X$, if there is function $\theta_{\xi,\om_X}\in C^\infty (X)$ unique up to adding a constant satisfying
$d\theta_{\xi,\om_X}(\cdot):=\om(\xi,\cdot)$, and
$\theta_{\xi,\om_X}$ is called a  {\red Hamiltonian} of $\xi$.
Let
$$\fham(X,\om_X):=\{ \xi\in \Ga(TX)\mid \xi \text{ is Hamiltonian} \}=\left\{\theta\in C^\infty(X)\left| \int_X\theta\om_X^n=0\right.\right\}$$
denotes the Lie algebra of the Hamiltonian diffeomorphism $\Ham(X,\om_X)<\Diff(X)$.
\end{defi}
To simplify the notation, from now on we will write $\om=\om_X$ if no confusion is caused. For any $g\in \Diff(X)$ and $\xi\in \fham(X,\om=\om_X)$,
we have
\begin{eqnarray}\label{th-om}
&&\om\Big( g_\ast\big(\xi|_{g^{-1}(x)}\big)\big|_x, \bullet|_x\Big)\nonumber=\om( g_\ast\big(\xi|_{g^{-1}(x)}\big)\big|_x, g_\ast\circ (g^{-1})_\ast(\bullet|_x))\\
&=&(g^\ast\om)\Big(\xi|_{g^{-1}(x)},\big((g^{-1})_\ast(\bullet|_x)\big)\big |_{g^{-1}(x)}\Big)
=(g^\ast\om)\Big(\xi|_{g^{-1}(x)},\big((g^{-1})_\ast(\bullet|_x)\big)\big |_{g^{-1}(x)}\Big)\nonumber\\
&=&d\theta_{\xi, g^\ast\om}\Big((g^{-1})_\ast(\bullet|_x)\big|_{g^{-1}(x)}\Big)
=d\Big((g^{-1})^\ast(\theta_{\xi, g^\ast\om})\Big)\big(\bullet|_x\big)
=d\Big(\theta_{\xi, g^\ast\om}\circ g^{-1}\Big)\Big |_x
\end{eqnarray}
from which we obtain an action
\begin{equation}
\begin{array}{ccccc}
\Diff_0(X)\times \fham(X,\om_X) & \xrightarrow{\hspace*{1.2cm}}  & \fham(X,\om_X) & \overset{\om^{-1}}{ \xrightarrow{\hspace*{1.2cm}} }&C^\infty(X) \\
(g,\xi) & \xmapsto{\hspace*{.6cm}}  & \Ad_g\xi:=g_\ast\big(\xi|_{g^{-1}(x)}\big)\big|_x  & \xmapsto{\hspace*{.6cm}}  & \Ad_g \theta_{\xi,\om}:=\theta_{\Ad_g\xi,\om}=(g^{-1})^\ast \big(\theta_{\xi,g^\ast \om}\big)
\end{array} .
\end{equation}
Clearly we have
$$\Ad_g \{\theta,\theta'\}_\om=\{\Ad_g\theta,\Ad_g\theta'\}_\om$$ as $\big(C_0^\infty(X),\{\bullet,\bullet\}_\om\big)\cong(\fham,[\bullet,\bullet])\subset(\fdiff,[\bullet,\bullet])$ is a {\red Lie subalgebra}.
As a consequence, we have
\begin{lemm}
With the notation introduced above, we have for $g\in \Diff(X)$
$$
\theta_{\Ad_g \xi,\om}=\theta_{\xi, g^\ast\om}\circ g^{-1} \ \text{ or equivalently }\ \ \theta_{\xi,g^\ast \om}=\theta_{\Ad_g \xi,\om}\circ g=g^\ast\theta_{g_\ast \xi,\om}.
$$
\end{lemm}

\begin{defi}\label{ham-J}
Let $J\in \cJ_\Int(X,\om_X)$ and $\xi\in \faut(X,J)$. We define the complexified $\Ham(X,\om_X)$-orbit of $J$ in $\cJ_\Int(X,\om_X)$ to be:
$$
\Ham^\CC(X,\om_X)\cdot J:=\bigcup_{g_s}
\left\{J_s:=J^{g_s}=(g_s^{-1})_\ast|_{g_s(x)}\circ (J|_{g_s(x)})\circ g_{s\ast}(\bullet |_x)\right \}_s
\subset\cJ_\Int(X,\om_X)
$$
with
\begin{equation}\label{ham-C}
\Ham^\CC(X,\om_X):=\left\{ g_1\in \Diff_0(X)\left |
 \{g_s\}_{s\in [0,1]}\subset \Diff_0(X) \text{ satisfying: }
\left\{
\begin{array}{lll}
 \dfrac{dg_s}{ds}&=&L_{\stackbin[\in \fham+{\red J_s}\circ\fham]{}{\underbrace{v_{\phi_0}+{\red J_s} v_{\phi_1}}}} \\
 g_0&=&\mathrm{id.}
\end{array}
\right.
\right.
\right\}.
\end{equation}
where $v_{\phi_i}, i=1,2$ are the Hamiltonian vector fields for $\phi_i:$
$
\left\{
\begin{array}{lll}
d\phi_i&=&\om(v_{\phi_i},\bullet); \\
 \phi_i &\in& C^\infty(X,\RR)
\end{array}.
\right.
$
Thus for $s\in [0,1]$, we have
$$ \om_{s}:=g_s^\ast \om \text{ satisfying }\dot{\om}_s\in\im\ii\dd_{J_s}\dbar_{J_s} \text{ and  } L_{\xi_s} J_s=0 \text{  for } \xi_s:=(g_s^{-1})_\ast \xi. \ \
\footnote{
Since for any $v\in \Ga(TX)$, we have
\begin{eqnarray}\label{L-xi-J}
\ \  \ \ \ \ \ \ (L_{\xi_s} J_s)\big|_x (g_s^{-1})_\ast \big( v|_{g_s(x)}\big)
&=&L_{\xi_s} \Big(J_s (g_s^{-1})_\ast ( v|_{g_s(x)})\Big)- J_s\Big(L_{\xi_s} \big( (g_s^{-1})_\ast ( v|_{g_s(x)})\big)\Big)\\
&=&L_{\xi_s}\Big( (g_s^{-1})_\ast|_{g_s(x)}\circ (J|_{g_s(x)})\circ g_{s\ast}\big( (g^{-1}_s)_\ast(v|_{g_s(x)})\big |_x\big) \Big)\nonumber\\
&&-(g_s^{-1})_\ast|_{g_s(x)}\circ (J|_{g_s(x)})\circ g_{s\ast}\left (\Big(L_{\xi_s} (g_s^{-1})_\ast ( v|_{g_s(x)})\Big)\Big |_x\right )\nonumber\\
&=&L_{(g_s^{-1})_{\ast} \xi}\Big( (g_s^{-1})_\ast|_{g_s(x)}\circ \big(J|_{g_s(x)}\cdot v|_{g_s(x)}\big)\Big)\nonumber\\
&&-(g_s^{-1})_\ast|_{g_s(x)}\circ (J|_{g_s(x)})\circ g_{s\ast}\left ((g_s^{-1})_\ast\Big(\big (L_{\xi} v\big)\big |_{g_s(x)}\Big)\Big |_x\right )\nonumber\\
&=&(g_s^{-1})_\ast|_{g_s(x)}\Big(\big(L_\xi (J\cdot v)\big |_{g_s(x)}\big)\Big)-(g_s^{-1})_\ast|_{g_s(x)}\Big(J|_{g_s(x)}\circ (L_{\xi} v)\big |_{g_s(x)}\Big)\nonumber\\
&=&(g_s^{-1})_\ast|_{g_s(x)}\Big(\big((L_\xi J) v\big)\big|_{g_s(x)}\Big).\nonumber
\end{eqnarray}
}
$$
\end{defi}
\begin{rema}
Notice that the Definition of $\Ham^\CC(X,\om_X)\cdot J$ above agree with Donaldson's description in \cite{Donaldson1997} (cf. Appendix \ref{D-94} ).
\end{rema}

\begin{defi}\label{fs}
Let
$$\om_s:=\om+ \ii\dd_J\dbar_J\phi_s=\om+\frac{\ii}{2} d\circ \Big(d-\ii J\circ d\Big)\phi_s=\om+{\red  \frac{1}{2}d\circ J\circ d\phi_s}, \ s\in [0,1]$$
be a smooth family of K\"ahler form on $X$ parametrized by $s\in [0,1]$
and let $f_s^\ast\om_s= \om$ with $\{f_s\}_{s\in [0,1]}\subset \Diff(X)_0$ be the  family of diffeomorphism obtained via Moser's trick:
 that is, if we let $f_s$ obtained by integrating  the vector field $\displaystyle v_s(x):=\left.\frac{d}{dt}\right|_{t=s} f_t\big(f_{-s}(x)\big)$, defined via the following:

\begin{eqnarray}\label{v-s-C}
0 &=&\frac{d (f_s^\ast\om_s)}{ds}=f_s^\ast \big(L_{v_s} \om_s\big)+f_s^\ast (\dot\om_s)=f_s^\ast \big(L_{v_s} \om_s\big)+ f_s^\ast \left(\frac{1}{2}d\circ J \circ d\dot \phi_s\right) \nonumber\\
&\Longrightarrow &
  \red v_s:=-\frac{1}{2}Jv_{\dot\phi_s,\om_s}-v_{\psi_s,\om_s} \
\end{eqnarray}
with
\begin{eqnarray*}
-{\red \om_s(v_{\psi_s,\om_s},\bullet)
=-d\psi_s:}
&=&v_s\lrcorner\om_s+\frac{1}{2}J\circ d\dot\phi_s=v_s\lrcorner\om_s{\red -}\frac{1}{2}d\dot\phi_s (J\bullet)\\
   \big({\red \ J\circ d\dot\phi_s= d\dot\phi_s(J^{-1}\bullet)= -d\dot\phi_s(J\bullet)} \ \big)
&=&\om_s(v_s,\bullet)-\frac{1}{2}\om_s(v_{\dot\phi_s,\om_s},J\bullet)=\om_s(v_s,\bullet)+\frac{1}{2}\om_s(Jv_{\dot\phi_s,\om_s},\bullet)\nonumber
\end{eqnarray*}
\end{defi}

Let $\xi^\CC:=\xi-\ii J\xi\in \Ga(T^{1,0}_JX)$ (i.e. $J\xi^\CC=\ii\xi^\CC$ ) be a {\red holomorphic vector field} satisfying:
$$L_{\xi^\CC} J=L_{\xi-\ii J\xi}J{\red :=}L_{\xi-J\circ J\xi} J=L_{2\xi}J=0$$
 and  we define $\theta_{\xi^\CC,\om}\in \theta_{\xi^\CC,\om}\in \fham^\CC(X,\om_X)\subset C^\infty(X,\CC)$ 
 \footnote{In particular, we have $\xi\in \fham(X,\om)+ J\cdot \fham(X,\om)$.
 }
 to be function ( unique up to a constant ) solving:
 $$\red \om(\xi^\CC,\bullet)=\dbar_J\theta_{\xi^\CC,\om}(\bullet).$$
\begin{rema}\label{xi-10}
If $\xi\in \Ga(TX)$ is a real Hamiltonian vector field satisfying: $\om(\xi,\bullet)=d\theta_{\xi,\om}(\bullet)$,  then
\begin{eqnarray*}
\om(\xi^{1,0},\bullet)
&=&\om\Big(\frac{\xi-\ii J\xi}{2},\bullet\Big)=\frac{1}{2}\Big(d\theta_{\xi,\om} (\bullet)-\ii\om(J\xi,\bullet)\Big)\\
&=&\frac{1}{2}\Big(d\theta_{\xi,\om}(\bullet)+\ii\om(\xi,J\bullet)\Big)=\frac{1}{2}\Big(d\theta_{\xi,\om}(\bullet)+\ii d\theta_{\xi,\om}(J\bullet)\Big)\\
\Big( J\circ d\theta(\bullet)=d\theta (J^{-1}\bullet) \Big)
&=&\frac{1}{2}\Big(d {\red -}\ii J\circ d\Big)\theta_{\xi,\om}(\bullet)=\dbar_J\theta_{\xi,\om} (\bullet)\red :=\dbar_J\theta_{\xi^{1,0},\om} (\bullet).
\end{eqnarray*}
\end{rema}
{\blue
In order to fit the $\Ham(X,\om_X)$-action on $\cJ_\Int(X,\om_X)$ into the story we developed in Section \ref{moment-convex}, e.g. Corollary \ref{chevalley},  we need a form of Chevalley isomorphism or a natural extension of $\Ad_{\Ham(X,\om_X)}$-invariant function $\fham(X,\om_X)\to \RR$ to a $\Ad_{\Ham^\CC(X,\om_X)}$-invariant function $\fham^\CC(X,\om_X)\to \CC$.  

Unfortunately, it is not possible to have such a statement. Even though we have defined $\Ham^\CC(X,\om_X)<\Diff(X)$, it is {\red NOT} a group complexifying the group $\Ham(X,\om_X)$ in the usual sense.  Luckily,   we do have a weaker version of Chevalley type result taking the form in Lemma \ref{Ad-th} combined which with Example \ref{trace} gives precisely the Tian-Zhu's generalized Futaki invariant.
}
\begin{lemm}\label{Ad-th}
Fix  
 a $J\in \cJ_\Int(X,\om=\om_X)$, let  $\xi^\CC=\xi_1+\ii\xi_2\in \fham^\CC(X,\om)$ be a {\red $J$-holomorphic vector field} satisfying $\red \om_s(\xi^\CC,\bullet)=\dbar_J\theta_{\xi^\CC,\om_s}(\bullet)$ with $\om_s=\om+ \ii\dd_J\dbar_J\phi_s $ and $\theta_{\xi^\CC,\om_s}$ be the unique $\CC$-valued function normalized via:
\begin{equation}\label{k}
\int_X\theta_{\xi^\CC,\om_s}\om_s^n=0.
\end{equation}
Let $\{f_s:=g_s^{-1}\}\subset \Diff(X)$ be defined in Definition \ref{fs}.  Then
\begin{enumerate}
\item $\theta_{\xi^\CC,\om_s}=\theta_{\xi^\CC,\om}+\ii (\xi^\CC\phi_s)$ 
\item
$$\frac{d}{ds}\Big(\stackbin[\red f_s^\ast\theta_{\xi^\CC,\om_s}]{}{\underbrace{(g_s^{-1})^\ast\theta_{\xi^\CC,\om_s}}}\Big)
=- \left\{f_s^\ast\theta_{\xi^\CC,\om_s}, f_s^\ast \left(\psi_s+\frac{\ii}{2}\dot\phi_s\right)\right\}_{\red f_s^\ast\om_s=\om}.
$$
\end{enumerate}
In particular, we have $f_s^\ast\theta_{\xi^\CC,\om_s}\in \fham(X,\om_X)$, i.e. $\displaystyle \int_X\big(f_s^\ast\theta_{\xi^\CC,\om_s}\big)\om_X^n=0$.
\end{lemm}
\begin{proof}
Let us introduce
\begin{eqnarray}
\bL_{\xi^\CC} J:&=&(d \circ \xi^\CC \lrcorner+\xi^\CC \lrcorner\circ d)\circ J
=\Big( d \circ \xi \lrcorner+\xi \lrcorner\circ d+\ii (d \circ (J\xi) \lrcorner+ (J\xi) \lrcorner\circ d)\Big)\circ J\\
&=&L_\xi J+\ii L_{J\xi}J
= L_\xi J-\ii \big(L_{\xi}J\big)\circ J=0\nonumber
\end{eqnarray}
where the last identity follows from \cite[Claim on page 36]{Tian2000} and our assumption $L_\xi J=0$. Moreover, one has 
$$\bL_{\xi^\CC}\circ d=d\circ \bL_{\xi^\CC}$$

\begin{eqnarray}\label{om-s-C}
\om_s(\xi^\CC,\bullet)
&=& \dbar_J\theta_{\xi^\CC,\om}+\frac{1}{2}\xi^\CC\lrcorner\circ d\circ (J\circ d\phi_s)\nonumber\\
&=& \dbar_J\theta_{\xi^\CC,\om}-\frac{1}{2}d \circ \xi^\CC \lrcorner\circ (J\circ d \phi_s)+\frac{1}{2}L_{\xi^\CC} (J\circ d\phi_s)\nonumber\\
\Big( J\circ d\phi_s(\bullet)=d\phi_s (J^{-1}\bullet) \Big)
&=& \dbar_J\theta_{\xi^\CC,\om}+\frac{1}{2}\Big(-d\big(d \phi_s(-J\xi^\CC)\big)+ ({\red \bL_{\xi^\CC} J})\circ d\phi_s+J\circ L_{\xi^\CC} (d\phi_s)\Big)\nonumber\\
(J\xi^\CC=\ii\xi^\CC;\  \bL_{\xi^\CC}\circ d=d\circ \bL_{\xi^\CC})
&=& \dbar_J\theta_{\xi^\CC,\om}+\frac{1}{2}\Big(\ii d\big(d \phi_s(\xi^\CC)\big)+ ({\red \bL_{\xi^\CC} J})\circ d\phi_s+J\circ  d(L_{\xi^\CC}\phi_s)\Big)\nonumber\\
&=& \dbar_J\theta_{\xi^\CC,\om}+\frac{\ii}{2}\stackbin[\red =2\dbar_J]{}{\underbrace{\Big(d-\ii J\circ d\Big)}}\big(\xi^\CC \phi_s\big)+ \frac{1}{2}({\red \bL_{\xi^\CC} J})\circ d\phi_s\nonumber\\
&=& \dbar_J\Big(\theta_{\xi^\CC,\om}+\ii(\xi^\CC \phi_s)\Big)+\frac{1}{2}\stackbin[\red =0]{}{\underbrace{(\bL_{\xi^\CC} J)}}\circ d\phi_s
\end{eqnarray}
The first part follows from the fact:
 
 \begin{eqnarray*}
  &&\frac{d}{ds}\int_X\Big(\theta_{\xi^\CC,\om}+\ii (\xi^\CC\phi_s)\Big)\om_s^n\\
 &=& \int_Xn\cdot \Big(\theta_{\xi^\CC,\om}+\ii (\xi^\CC\phi_s)\Big)\cdot \ii\dd_J\dbar_J\dot\phi_s\wedge\om_s^{n-1}+\ii(\xi^\CC\dot\phi_s)\om_s^n\\
 &=&\ii  \int_X \big(n\cdot\dbar_J\Big(\theta_{\xi^\CC,\om}+\ii (\xi^\CC\phi_s)\Big)\wedge\dd_J\dot\phi_s\wedge\om_s^{n-1}+(\xi^\CC\dot\phi_s)\om_s^n\big)\\
 &\stackrel{\eqref{om-s-C}}{=}&\ii  \int_X \big(-\dd_J\dot\phi\wedge\big(\xi^\CC\lrcorner\om_s^{n}\big)+(\xi^\CC\dot\phi_s)\om_s^n\big)
 =\ii \int_X \left(-\big(\xi^\CC\lrcorner\dd_J\dot\phi_s\big)\cdot\om_s^{n}+(\xi^\CC\dot\phi_s)\om_s^n\right)=0.\\
 \end{eqnarray*}
Now for the second part, we have
\begin{eqnarray}\label{theta-inv-C}
&&\frac{d}{ds}\Big( f_s^\ast\theta_{\xi^\CC,\om_s}\Big)\nonumber\\
&\stackrel{}{=}& f_s^\ast\Big(L_{\red v_s}\theta_{\xi^\CC,\om_s}\Big)+f_s^\ast\left( \frac{d\theta_{\xi^\CC,\om_s}}{ds}\right)\nonumber\\
&\stackrel{\eqref{v-s-C}}{=}& -f_s^\ast\Big(L_{\red  \frac{1}{2}J v_{ \dot\phi_s,\om_s}+v_{\psi,\om_s}}\theta_{\xi^\CC,\om_s}\Big)+f_s^\ast\left( \frac{d\theta_{\xi^\CC,\om_s}}{ds}\right)
\nonumber\\
&\stackrel{\eqref{om-s-C}}{=}&
-f_s^\ast\Big(L_{ \red \frac{1}{2}J v_{\red\dot\phi_s,\om_s}+v_{\psi,\om_s}}\theta_{\xi^\CC,\om_s}\Big)+{\red \ii} f_s^\ast\big( \xi^\CC(\dot\phi_s)\big)\nonumber\\
&=&-f_s^\ast\Big(d\theta_{\xi^\CC,\om_s}\left({\red \frac{1}{2}J v_{\dot\phi_s,\om_s}+v_{\psi_s,\om_s}}\right)\Big)+f_s^\ast\left(d\dot\phi_s(J\xi^\CC)\right)\nonumber\\
&=&-\frac{1}{2}f_s^\ast\Big((\dd_J\theta_{\xi^\CC,\om_s}+\dbar_J\theta_{\xi^\CC,\om_s})({\red J v_{\dot\phi_s,\om_s}})\Big)+f_s^\ast\big( (J\xi^\CC)(\dot\phi_s)\big)
-f_s^\ast\big( d\theta_{\xi^\CC,\om_s}( v_{\psi_s,\om_s})\big)\nonumber\\
\label{theta-inv-2}
&\stackrel{\eqref{d-d}}{=}&- \left\{f_s^\ast\theta_{\xi^\CC,\om_s}, f_s^\ast \left(\psi_s+\frac{\ii}{2}\dot\phi_s\right)\right\}_{\red f_s^\ast\om_s=\om}
\end{eqnarray}
thanks to the following:
\begin{eqnarray}\label{d-d}
&&-(\dd_J\theta_{\xi^\CC,\om_s}+\dbar_J\theta_{\xi^\CC,\om_s})\big({\red J v_{\dot\phi_s,\om_s}}\big)+2 (J\xi^\CC)(\dot\phi_s)\nonumber\\
&=&-\Big(\dd_J\theta_{\xi^\CC,\om_s}+\dbar_J\theta_{\xi^\CC,\om_s}\Big)\big({\red J v_{\dot\phi_s,\om_s}}\big)+2 d\dot\phi_s(J\xi^\CC)\nonumber\\
&=&-\Big(\dd_J\theta_{\xi^\CC,\om_s}+\dbar_J\theta_{\xi^\CC,\om_s}\Big)\big({\red J v_{\dot\phi_s,\om_s}}\big)+2 \om_s(\xi^\CC,Jv_{\dot\phi_s,\om_s})
\nonumber\\
&=&-\Big(\dd_J\theta_{\xi^\CC,\om_s}{\red -}\dbar_J\theta_{\xi^\CC,\om_s}\Big)\big({\red J v_{\dot\phi_s,\om_s}}\big)\nonumber\\
\Big((\dd_J-\dbar_J)=\ii J\circ d\Big)
&=&-\ii J\circ d\theta_{\xi^\CC,\om_s}\big({\red J v_{\dot\phi_s,\om_s}}\big)=-\ii d\theta_{\xi^\CC,\om_s}\big({\red  v_{\dot\phi_s,\om_s}}\big)\nonumber\\
&=&-\ii\{\theta_{\xi^\CC,\om_s},\dot\phi_s\}_{\om_s}.
\end{eqnarray}
\end{proof}
\begin{rema}
If $\xi^\CC=\xi^{1,0}$ with $\xi$ being a {\red Killing vector field} than we have 
$\theta_{\xi^{1,0},\om_s}=\theta_{\xi,\om_s}$ thanks to Remark \ref{xi-10} and $\{\theta_{\xi,\om_s},\dot\phi_s\}_{\om_s}=\xi(\dot\phi_s)=0$.
This together with \eqref{th-om} and the fact:
$d\Big( (g_s^{-1})^\ast \theta_{\xi,g_s^\ast\om}\Big)(\bullet|_x)=\om\Big( g_{s\ast}\big(\xi|_{g_s^{-1}(x)}\big),\bullet|_x\Big)$
imply
\begin{eqnarray*}
\frac{d \Big( g_{s\ast}\big(\xi|_{g_s^{-1}(x)}\big)\Big)}{ds}
&=& g_{s\ast}\Big( L_{Jv_{\dot\phi_s,\om_s}}\xi\Big)=g_{s\ast}\Big(-L_\xi \big(J\cdot v_{\dot\phi_s,\om_s}\big)\Big)
=g_{s\ast}\Big(-(L_\xi J)\cdot v_{\dot\phi_s,\om_s}-J\cdot\big( L_\xi  v_{\dot\phi_s,\om_s}\big)\Big)\\
({\red \because  L_\xi J=0}\ )
&=&g_{s\ast}\Big(-J\big(L_\xi v_{\dot\phi_s,\om_s}\big)\Big)=g_{s\ast}\Big(-J v_{\xi(\dot\phi_s),\om_s}\Big)=0. \ \   \big( {\red \because \xi(\phi_s)=0}\big )
\end{eqnarray*}

\end{rema}

\begin{coro}\label{Fut-F}
Fix a  $ F(t)\in \CC[\![t]\!]$ or $C^\infty(\CC)$, let $\{f_s=g_s^{-1}\}_s<\Diff_0(X)$, $J\in \cJ_\Int(X,\om=\om_X)$, $\xi^\CC,\eta^\CC\in H^0(T^{1,0}_JX)$, i.e. $L_{\xi^\CC} J=0, L_{\eta^\CC} J=0$ and $\theta_{\xi^\CC,\om_s},\theta_{\eta^\CC,\om_s}$ defined in Lemma \ref{Ad-th}.
 Then 
 \begin{enumerate}
 \item
$\displaystyle \int_X F(\theta_{\xi^\CC,\om_s})\om_s^n$ is {\red independent} of $s$.
 \item 
If we write
$
 \theta_{\bullet,s}:= f_s^\ast\theta_{\bullet,\om_s}=\stackbin[\red f_s^\ast\theta_{\bullet,\om_s}]{}{\underbrace{(g_s^{-1})^\ast\theta_{\bullet,\om_s}}}.
$
Then
$$
\int_X \left.\frac{d}{d\ep}\right|_{\ep=0} F(\theta_{\xi^\CC+\ep \eta^\CC,s})\om_s^n=\int_X \theta_{\eta^\CC,s}\cdot F'(\theta_{\xi^\CC,s})\om_s^n.
$$
is independent of $s$.
 \end{enumerate}
\end{coro}
\begin{proof}


Let $\vphi_s^\CC:=\psi_s+\frac{\ii}{2}\dot\phi_s$, then the first part follows:
\begin{eqnarray}\label{dot-C-0}
&&\frac{d}{ds}\left(\int_X F(\theta_{\xi^\CC,\om_s})\om_s^n\right)=\frac{d}{ds}\left(\int_X F(\theta_{\xi^\CC,\om_s})g^\ast_s\om^n
\right)=\frac{d}{ds}\left(\int_X F((g_s^{-1})^\ast\theta_{\xi^\CC,\om_s})\om^n\right)\\
&=&\!\!\int_X F'(f_s^\ast\theta_{\xi^\CC,\om_s})\left(- \left\{f_s^\ast\theta_{\xi^\CC,\om_s}, f_s^\ast \left(\varphi_s^\CC\right)\right\}_{\red f_s^\ast\om_s=\om} \right)\om^n\nonumber\\
&=&\!\!\int_X \left\{ F\big (f_s^\ast\theta_{\xi^\CC,\om_s}\big), f_s^\ast \left(\varphi_s^\CC\right)\right\}_{\red f_s^\ast\om_s=\om}
\!\!\!\!\!\!\!\om^n
= -\stackbin[\red=0]{}{\underbrace{\int_X \left( dF(\theta_{\xi^\CC,\om_s})\wedge d\left(\varphi^\CC_s\right)\right)\om^n}}=0.\nonumber
\end{eqnarray}

and the second part follows from the fact:
 \begin{eqnarray*}
 \frac{d}{ds}\left(\theta_{\eta^\CC,s} F'(\theta_{\xi^\CC,s})\right)
 &=&\frac{d\theta_{\eta^\CC,s}}{ds}F'(\theta_{\xi^\CC,s})+\theta_{\eta^\CC,s} \frac{d F'(\theta_{\xi^\CC,s})}{ds}\\
 &=&\left\{\vphi^\CC, \theta_{\eta^\CC,s}\right\}F'(\theta_{\xi^\CC,s})+\theta_{\eta^\CC,s} \left\{\vphi^\CC, F'(\theta_{\xi^\CC,s})\right\}\\
 &=&\left\{\vphi^\CC,\theta_{\eta^\CC,s}\cdot F'(\theta_{\xi^\CC,s})\right\}.
\end{eqnarray*}

\end{proof}


\begin{exam}
Consider the Example \ref{F-check}, with $F(t)=e^t-t$ then we have
\begin{eqnarray}\label{exp-x}
\left.\frac{d}{d\ep}\right|_{\ep=0} F(\theta_{\xi^\CC+\ep \eta^\CC,s})=\frac{d}{d\xi^\CC}\left(e^{\theta_{\xi^\CC+\ep\eta^\CC,s}}-\theta_{\xi^\CC+\ep\eta^\CC,s}\right)=\left(\theta_{\eta^\CC,s}\big(e^{\theta_{\xi^\CC,s}}-1\big)\right).
\end{eqnarray}
\end{exam}


\subsection{ Functional on $\fham(X,\om_X)$.}
In this section, we study an $\infty$-dimensional example of Section \ref{moment-convex}. We introduce a special type of $\Ad_{\Ham(X,\om_X)}$-invariant functional on $\fham(X,\om_X)$.

\begin{lemm}\label{bD-bD'}    
Let $F(x)\in \RR[\![x]\!]$ and
we define functional
 \begin{equation}\label{D-fcn}
\bD=\bD_F:
\begin{array}{cccc}
\mathfrak{ham}(X,\om_X)^\ast=\wedge^{2n}_X & \overset{}{ \xrightarrow{\hspace*{1.cm}}} & \RR  \ \\
  u\cdot \dfrac{\om^n}{n!}& \longmapsto &\displaystyle\int_X F(u)\dfrac{\om^n}{n!} \\
\end{array}.
\end{equation}
Then 
\begin{equation}
d \bD\big|_{u\om^n/n!}:
\begin{array}{ccccc}
\fham(X,\om_X)^\ast=\wedge^{2n}_X & \xrightarrow{\hspace*{1.6cm}}  & \fham(X,\om_X) \\
 u\cdot\dfrac{\om_X^n}{n!}& \xmapsto{\hspace*{1.0cm}}  & F'(u)-\displaystyle\frac{1}{V}\int_X F'(u)\frac{\om_X^n}{n!}&
\end{array}
\end{equation}
\begin{equation}
\nabla^2 \bD\big|_{u\om^n/n!}:
\begin{array}{ccccc}
\fham(X,\om_X)^\ast=\wedge^{2n}_X & \xrightarrow{\hspace*{1.6cm}}  & \fham(X,\om_X) \\
(\de u)\cdot\dfrac{\om^n}{n!} & \xmapsto{\hspace*{1.0cm}}  & F''(u)\cdot \de u-\displaystyle\frac{1}{V}\int_XF''(u)\cdot \de u\frac{\om_X^n}{n!}&
\end{array}
\end{equation}
where we use identification $\fham(X,\om_X)^*=\fham(X,\om_X)\cdot \dfrac{\om^n_X}{n!}$.
\end{lemm}



\subsection{K\"ahler-Ricci soliton.}
In this subsection, following \cite{Donaldson2017} we will interpret K\"ahler-Ricci soliton equation as a {\red $\bD$-extremal point}
(in the sense of Section \ref{moment-convex}) in $\cJ_\Int(X,\om_X)$ with respect to the $\Ham(X,\om_X)$-action.  To achieve that, let us recall 
that $(L,h,\nabla)\to (X,\om_X=\Ric(\nabla))$ is a symplectic manifold  equipped with a prequantum line bundle $(L,h)$ with Hermtian metric $h$ with $\om=\Ric(h)$. Assume further that
$c_1(L)=c_1(X,\om_X)$, then for a fixed $J\in \cJ_\Int(X,\om_X)$ the space of 
$$\ker (d^\nabla\big|_J):=\{\al\in \Ga(\wedge^{n,0}_J(X)\otimes L)\mid d^\nabla \al=0\}\subset \Ga(\wedge^{n,0}_J(X)\otimes L)$$
is of $\dim=1$ and $\ker(d^\nabla)=\CC\cdot \al$.  In particular,
$N:=\ker(d^\nabla)\to \cJ_\Int(X,\om_X)$ is a line bundle over $\cJ_\Int(X,\om_X)$,
\begin{equation}
\xymatrix{                        
&N \ar@{^{(}->}[rr]^{}   \ar@{>}[d]^{\pi_\cJ}  && \pi_\cJ (K_\cX\otimes L)\ar@{>}[d]^{\pi}  &  \Ga(\wedge_J^{n,0}(X)\otimes L)\ar@{_{(}->}[l]^{} \ar@{>}[d]^{\pi} \\
  &            \cJ_\intg(X,\om_X)\ar@{^{(}->}[rr]^{}  && \cJ_\intg(X,\om_X) &  J \ar@{_{(}->}[l]^{}
            }
\text{ with }
\xymatrix{                        
\cX:=X\times\cJ_\Int(X,\om_X)  \ar@{>}[d]^{\pi_\cJ}  &  (X,J)  \ar@{_{(}->}[l]^{} \ar@{>}[d]^{\pi}. \\
 \cJ_\intg(X,\om_X) &  J \ar@{_{(}->}[l]^{}
            }
\end{equation}
For a local section  $\al\in \Ga(N)$ near $J\in \cJ_\Int(X,\om_X)$, Donaldson (cf. \cite[Theorem 1]{Donaldson2017}) defines the quadratic pairing:
$$
\lla\be,\be\rra_\cJ:=-\lla\be,\be\rra|_J+\frac{|\lla\al,\be\rra_J|^2}{\lla \al,\al\rra_J} \text{ with } \be:=\de\al \text{ and }
\displaystyle\lla \be,\be\rra|_J:=\int_X\be\wedge_h\bar \be \text{ for } \be\in \wedge_J^{n,0}(X)\otimes L
$$
which descends to a $\Ham(X,\om_X)$-invariant K\"ahler metric on $\cJ_\Int(X,\om_X)$, called {\red Berndtsson metric} by Donaldson.
Moreover, it was proved in \cite[Proposition 1] {Donaldson2017} that the moment map for $\Ham(X,\om_X)$-action on 
$(Z,\om)=\big(\cJ_\Int(X,\om_X), \lla\bullet,\bullet\rra_\cJ\big)$
is 
$$
\mu(J)=\mu([\al(J)]):=\frac{\al(J)\wedge_h\bar \al(J)}{\lla \al(J),\al(J)\rra}-\frac{\om_X^n}{V\cdot n!}\in\fham(X,\om_X)^\ast\  \text{ with } V=\int_X\frac{\om_X^n}{n!} \text{ and } \al\in N|_J\subset\Ga(\wedge^{n,0}_J(X)\otimes L).
$$
Let us introduce function $\phi=\phi_J\in C^\infty(X)$ via:
$$
\Om(J):=\frac{\al(J)\wedge_h\bar \al(J)}{\lla \al(J),\al(J)\rra}=\frac{e^{\phi_J}\cdot \om_X^n}{Vn!} \text{ with normalization } \frac{1}{V}\int_X(e^{\phi_J}-1)\frac{\om_X^n}{n!}=0
$$
then
\begin{equation}\label{mu-J}
\mu(J)=\frac{\al(J)\wedge_h\bar \al(J)}{\lla \al(J),\al(J)\rra}-\frac{\om_X^n}{V\cdot n!}= (e^{\phi_J}-1)\frac{\om_X^n}{Vn!}
\text{ and } \Ric(\Om(J))=\om_X
\end{equation}
where we regard $\Om(J)$ is a metric on the canonical line bundle $\wedge^{n,0}_J(X)^\ast\to X$. 
Let $\xi$ be the Hamiltonian Killing vector field and let
$$
d\hat\theta_\xi=\om_X(\xi,\cdot) \text{ satisfying }\int_X(e^{\hat\theta_\xi}-1)\frac{\om_X^n}{Vn!}=0, \ 
\footnote{ The normalization of $\theta_\xi$ is chosen so that there is no extra constant  in the  K\"ahler-Ricci soliton $\phi=\theta_\xi$.}
$$
then associated K\"ahler-Ricci soliton equation  with K\"ahler-Ricci soliton vector field $\xi$ is:
\begin{equation}\label{phi-th}
\frac{\al(J)\wedge_h\bar \al(J)}{\lla \al(J),\al(J)\rra}=\Om(J)=\frac{e^{\hat\theta_\xi}\cdot \om^n}{Vn!}  {\red \xLeftrightarrow{\hspace*{1.2cm}}}
\left\{
\begin{array}{rcl}
\phi_J&=&\hat\theta_\xi; \\
 & & \\
\displaystyle\int_X(e^{\phi_J}-1)\frac{\om_X^n}{Vn!}&=&\displaystyle\int_X(e^{\hat\theta_\xi}-1)\frac{\om_X^n}{Vn!}=0.
\end{array}
\right.
 \end{equation}
Following \eqref{dg}, the negative gradient flow for the functional $\bD_F:\fham(X,\om_X)^\ast=\wedge^{2n}(X)\to \RR$ is given  by
 \begin{equation}\label{KR-grad}
\dot\vphi-\dfrac{1}{V}\int_X\dot\vphi\frac{\om_X^n}{n!}=d\bD_F(J)=F'(\mu(J))-\dfrac{1}{V}\int_XF'(\mu(J))\dfrac{\om_X^n}{n!}=\phi_J-\dfrac{1}{V}\int_X\phi_J\frac{\om_X^n}{n!}\in \fham(X,\om_X)
 \end{equation}
 where $ \vphi\in \PSH(X,\om_X)$. On the other hand, the  K\"aher Ricci flow is given by
 \begin{eqnarray*}
 \ii\dd_J\dbar_J\dot\vphi&=&\dot \om=\ii\dd_J\dbar_J\log\frac{\Om(J)}{\om_X^n/n!}=\ii\dd_J\dbar_J  \log\left(\frac{\Om(J)}{\om_X^n/n!}-\frac{1}{V}+\frac{1}{V}\right)\cdot V\\
 &=&\ii \dd_J\dbar_J\log\Big(\stackbin[u:=]{}{\underbrace{\frac{V\mu(J)}{\om_X^n/n!}}}+1\Big) 
 \text{ by Definition \eqref{D-fcn}}.
 \end{eqnarray*}
comparing with \eqref{KR-grad}, we deduce that $F'(u)=\log(Vu+1)$ with $u=\dfrac{\mu(J)}{\om_X^n/n!}$, hence 
 $$F''(u)=\dfrac{1}{u+1/V} \text{ and  } F(u)=(u+1/V)\log(Vu+1)-u+{\rm const.}\ .$$

By applying Theorem \ref{Calabi} to $(Z,\om)=(\cJ_\Int(X,\om_X),\lla \bullet,\bullet\rra_\cJ)$ equipped with $\Ham(X,\om_X)$-invariant convex functional $\bD_F$ we recover \cite[Theorem A]{TianZhu2000} as follows
\begin{coro} Let us continue with the notation introduced as above and assume further that $(X,\om_X,J)$ is a Fano manifold admitting a K\"ahler-Ricci soliton with the {\red K\"ahler-Ricci soliton vector field}, equivalently,  the {\red $\bD$-extremal vector field},  $\xi\in\faut:= \faut(X, \om_X,J)$. Then
we have the Calabi decomposition:
$$
\faut=\faut_0\oplus\bigcup_{\lam >0}\faut_\lam \text{ with } \faut_\lam:=\left\{\eta^\CC\in \faut\left| [\xi,\eta^\CC]=\lam \eta^\CC\right.\right\}.
$$
\end{coro}


\subsection{Tian-Zhu's generalized Futaki-invariant }

Next we discuss how the generalized Futaki invariant firstly introduced in \cite{TianZhu2000} fits into our framework developed in Section \ref{moment-convex}. Let us assume $(X,\om_X,J)$ is a Fano manifold as before.
To apply Theorem \ref{Fut}  and Example \ref{trace} to the situation here, we need to establish the $\Ad_{\Ham^\CC(X,\om_X)}$-invariance of
\begin{equation}\label{bD-th}
\la d\stackbin[\in \fham^\ast]{}{\underbrace{\bD^{-1}(\theta_{\xi^\CC})}},\theta_{\eta^\CC}\ra \text{ for }\theta_{\xi^\CC},\theta_{\eta^\CC}\in \faut(X,\om_X, J)\subset \fham(X,\om_X) \text{ at }J\in \cJ(X,\om_X).
\end{equation}
For this, we have the following statement parallel to Corollary \ref{chevalley}:
\begin{prop}
Suppose $F\in C^\infty(\RR)$ is {\red convex} with $ \widehat F\in C^\infty(\im F',\RR)$ being the Legendre transform. Let $\bD$ be defined as in Lemma \ref{bD-bD'} such that  its Legendre transform satisfies(cf. Section \ref{Leg}):
$$
\widehat {\bD_F}(\phi)=\bD_{\widehat F}(\phi):
\begin{array}{ccccc}
C^\infty(X) & \xrightarrow{\hspace*{1.6cm}}  & \RR \\
\phi & \xmapsto{\hspace*{1.0cm}}  &\displaystyle \frac{1}{V} \int_X \widehat F (\phi)\frac{\om_X^n}{n!} &.
\end{array}
$$
Then 
\begin{enumerate}
\item The negative gradient flow equation is given by 
$$
\dot \phi=d\bD_{\widehat F}(\phi).
$$
\item

for $\theta_{\xi^\CC},\theta_{\eta^\CC}\in \faut(X,\om,J)$
$$
\la d\stackbin[\in \fham^\ast]{}{\underbrace{(\bD_F)^{-1}(\theta_{\xi^\CC})}},\theta_{\eta^\CC}\ra=\la d\stackbin[\in \fham^\ast]{}{\underbrace{\bD_{\widehat F}(\theta_{\xi^\CC})}},\theta_{\eta^\CC}\ra
$$
is constant along the $\Ham^\CC(X,\om)$-orbit of $J\in \cJ(X,\om_X)$. In particular, the $\mu$-invariant in this case is given by 
$\left\la d\bD_{\widehat F}(\hat\theta_\xi)- \mu(J),\theta_\eta\right\ra$.
\end{enumerate}
\end{prop}
\begin{proof}
By Proposition \ref{D-leg}, we have
\begin{eqnarray*}
\la d\stackbin[\in \fham^\ast]{}{\underbrace{\bD^{-1}(\theta_{\xi^\CC})}},\theta_{\eta^\CC}\ra
=\la d\stackbin[\in \fham^\ast]{}{\underbrace{\widehat\bD(\theta_{\xi^\CC})}},\theta_{\eta^\CC}\ra
&=&\frac{1}{V}\int_X\widehat F'(\theta_{\xi^\CC})\cdot\theta_{\eta^\CC}\frac{\om_X^n}{n!}.\\
\end{eqnarray*}
Then our claim follows from Corollary \ref{Fut-F} 
and the assumption that $\theta_{\xi^\CC},\theta_{\eta^\CC}\in \faut(X,\om_X,J)$.

\end{proof}
\begin{rema}
Note that $\widehat\bD$ is {\red not } a $\Diff(X)$-invariant functional.
\footnote{
For any $\phi\in \Diff(X)$ the functional
\begin{equation}
\begin{array}{ccccc}
\wedge^{2n}_X & \xrightarrow{\hspace*{1.6cm}}  & \RR \\
\phi^\ast (g\cdot \om^n/n!) & \xmapsto{\hspace*{1.0cm}}  &\displaystyle \int_X F\left(\phi^\ast g\cdot\Big(\frac{ \phi^\ast\om^n}{\om^n}\Big)\right)\cdot\frac{\om^n}{n!}&.
\end{array}
\end{equation}
one has
$$
\displaystyle \int_X F\left(\phi^\ast g\cdot\Big(\frac{ \phi^\ast\om^n}{\om^n}\Big)\right)\cdot\frac{\om^n}{n!}\ne \displaystyle \int_X F\left(\phi^\ast g\right)\cdot\frac{\phi^\ast\om^n}{n!}
$$
for general $\phi\in \Diff(X)$.
}
\end{rema}
To recover Tian-Zhu's generalized Futaki-invariant, let us fix some notation:
$$
\Ric(\om_X)-\om_X=\ddbar \phi \text{ with normalization } \frac{1}{V}\int_X (e^\phi-1)\frac{\om_X^n}{n!}=0.
$$ 
For any $\xi,\eta \in \faut(X,\om_X,J)$, we define normalizations:

\begin{equation}\label{th-ti-th}
d\theta_\eta(\bullet)=d\hat\theta_\eta(\bullet)=d\widetilde\theta_\eta(\bullet)=\om_X(\eta,\bullet) 
\text{ with normalization }
0=
\left\{
\begin{array}{lc}
 \displaystyle\int_X\theta_\eta \frac{\om^n_X}{n!} &\\
 \displaystyle\int_X\left(e^{\hat\theta_\eta}-1\right)\frac{\om_X^n}{n!}&\\
 \displaystyle\int_X\widetilde{\theta_\eta}\cdot e^\phi\frac{\om_X^n}{n!} & (\text{ cf. \cite[Section 2]{TianZhu2000}})\\
\end{array}
\right.
\end{equation}
\begin{prop}\label{tz-D}
The Tian-Zhu generalized invariant is given by 
$$
\frac{\int_X\widetilde\theta_\eta\cdot e^{\widetilde\theta_\xi}\cdot\om_X^n}{\int_X e^{\widetilde{\theta_{\xi}}}\cdot\om_X^n}
=
\left  \la d\widehat{\bD_F}(\hat\theta_{\xi})-\mu(J),\ \theta_\eta\right\ra
\text{ with }
\left\{ 
\begin{array}{ccll}
\mu(J)&:=&\displaystyle\frac{1}{V}(e^\phi-1)\frac{\om_X^n}{n!} & (\text{ cf. \ref{mu-J} })\\
 & & &\\
d\widehat{\bD_F}(u)&:=&\dfrac{1}{V}(e^u-1)\dfrac{\om_X^n}{n!} & (\text{ cf. Example \ref{F-check} })
\end{array}
\right.
$$
\end{prop}
\begin{proof}
By Definition, we have 
$
\widetilde{\theta_{\eta}}=\theta_{\eta}-C_\eta.
$
Using 
$$
\left.
\begin{array}{rcl}
0=\displaystyle\int_X\widetilde{\theta_\eta}\cdot e^\phi\frac{\om_X^n}{n!}&=&\displaystyle\int_X(\theta_\eta-C)\left(\mu(J)+\frac{\om_X^n}{Vn!}\right)\\
& & \\
0&=&\displaystyle\int_X\theta_\eta \frac{\om^n_X}{n!}
\end{array}
\right\}
{\red \xRightarrow{\hspace*{0.7cm}}}\ 
\left\{\!\!\!
\begin{array}{rcl}  
C_\eta\!\!\!
&=&\!\!\!\displaystyle\frac{C_\eta}{V}\int_X e^\phi\cdot\frac{\om_X^n}{n!}=\displaystyle\int_X\theta_\eta\cdot \mu(J)+\stackbin[\red =0]{}{\underbrace{\frac{1}{V}\int_X\theta_\eta\frac{\om_X^n}{n!}}}\\
&=&\!\!\!\displaystyle\int_X\theta_\eta\cdot \mu(J)=\la\mu,\theta_\eta\ra
\end{array}
\right.
$$
we obtain that $\widetilde{\theta_{\eta}}=\theta_{\eta}-\la\mu(J),\theta_\eta\ra$.
Let us introduce $\hat C$ via:
$$
\int_X e^{\widetilde{\theta_{\xi}}}\omega_X^n=\int_X e^{{\hat\theta_{\xi}}}e^{\hat C}\omega_X^n \text{ with }
\frac{1}{V}\int_X \left(e^{\hat\theta_\xi}-1\right)\frac{\om_X^n}{n!}=0
\text{ then } e^{\hat C}=\frac{\displaystyle\int_X e^{\widetilde\theta_\xi}\om_X^n}{\displaystyle\int_X e^{\hat\theta_\xi}\om_X^n}
=\frac{1}{V}\displaystyle\int_X e^{\widetilde\theta_\xi}\om_X^n.
$$
Then
\begin{eqnarray*}
\frac{1}{V} \int_X\widetilde{\theta_{\eta}} e^{ \widetilde{\theta_{\xi}}}\frac{\omega_X^n}{n!}
&=& \frac{1}{V} \int_X (\theta_{\eta}-\langle \mu(J), \theta_\eta\rangle)e^{\hat C} e^{ \hat\theta_{\xi}}\frac{\omega_X^n}{n!}\\
&=&\frac{e^{\hat C}}{V} \int_X (\theta_{\eta}-\langle \mu(J),\theta_\eta\rangle)\left(e^{ \hat\theta_{\xi}}-1+1\right)\frac{\omega_X^n}{n!}\\
\left(\because \frac{1}{V}\int_X \left(e^{\hat\theta_\xi}-1\right)\frac{\om_X^n}{n!}=0 \right)
&=&e^{\hat C} \left(\frac{1}{V} \int_X \theta_{\eta}\cdot \left(e^{\hat \theta_{\xi}}-1\right)\frac{\om_X^n}{n!}-
\la \mu(J),\theta_\eta\ra\right)\\
(\because \eqref{th-ti-th})
&=&e^{\hat C}\cdot\left( \la d\widehat{\bD_F}(\hat\theta_\xi),\theta_\eta\ra- \la \mu(J),\theta_\eta\ra\right)\\
&=&e^{\hat C}\cdot
\left\la d\widehat{\bD_F}(\hat\theta_\xi)- \mu(J),\theta_\eta\right\ra\\
\end{eqnarray*}
Notice that the normalizations of  $\displaystyle \frac{1}{V}\int_X \left(e^{\hat\theta_\xi}-1\right)\frac{\om_X^n}{n!}=0$ and $\displaystyle\int_X\theta_\xi\om^n_X=0$ imply
$$
d\widehat{\bD_F}(\hat\theta_\xi)=\frac{e^{\hat \theta_\xi}-1}{V}=\frac{e^{\theta_\xi}}{\displaystyle\frac{1}{V}\int_X e^{\theta_\xi}\frac{\om_X^n}{n!}}-\frac{1}{V}
$$

\end{proof}

%

\section{Appendix}
\subsection{Extending $\Ad_K$-invariant functions}\label{ext}
As we mentioned in Remark \ref{gen-Chevalley}, the $\Ad_G$-extension of $f\in C^\infty(\fk)^{\Ad_K}$  is non-unique and exists in quite general sense contrast to the canonical Chevalley isomorphism. 
\begin{prop}
Let $f\in C^\infty(\fk)^{\Ad_K}$ is a restriction of a function, which by abusing of notation still denoted by $f\in C^\infty(\fg)^{\Ad_G}$.

\end{prop}

\begin{proof}
For simplicity, we will specialize to the case that $K=\UU(n)$, $G=\GL(n,\mathbb{C})$ and $f:\mathfrak{k}\rightarrow\mathbb{R}$ a smooth $\Ad_K$ invariant function. Then $f$ is the restriction of a smooth $\Ad_G^*$ invariant function $F:\mathfrak{g}^*\rightarrow \mathbb{R}$.


\vskip .1in

\begin{lemm}\label{conj} $A,B\in \mathfrak{k}$ and assume that $A = gBg^{-1}$ for some $g\in G$. Then $A=kBk^{-1}$ for some $k\in K$.
\end{lemm}

Let $f:\fk\rightarrow\mathbb{R}$ a smooth $\Ad_K$ invariant function. Define a smooth function $\ti f: (\ii\RR)^n\rightarrow\RR$
by the formula $\ti f(\lambda_1,...,\lambda_n)= f(\xi)$ where $\xi\in \fk$ is any element whose characteristic polynomial is $\prod_{j=1}^n (x-\lam_j)$.
 This is well defined since Lemma \ref{conj} implies that if $\xi,\xi'\in \fk$ have the same characteristic polynomial, then they are conjugate by an element in $K$ and $f$ is $\Ad_K$ invariant.
\vskip .1in

The function $\tilde f$ is smooth and invariant under the permutation group $S_n$. Let $\tilde F:\CC^n\rightarrow\RR$ be any smooth $\mathfrak S_n$-invariant extension of $f$, for example, $\tilde F(\lambda_1,...,\lambda_n)=
\tilde f(\ii{\rm Im}(\lam_1),...,\ii{\rm Im}(\lam_n))$. Let us define an extension of $f$, which by abusing of notation still denoted by 
$f(\xi):= \tilde F(\lam_1,...,\lam_n):\fg\rightarrow\RR$ where $\prod_{j=1}^n (x-\lam_j)$ is the characteristic polynomial of $\xi\in\fg$. Note $f$ is well defined since $\ti F$ is invariant under the symmetric group $\mathfrak S_n$. Moreover $f|_\fg$ is $\Ad_G$ invariant since all the elements in an $\Ad_G$ orbit have the same characteristic polynomial and it agrees with $f$ on $\fk$ by construction.
\end{proof}

\subsection{Idendification of $\PSH(X,\om_X)=\Ham(X,\om_X)^\CC/\Ham(X,\om_X)$}\label{D-94}
Let $\om_s:=\om_{\phi_s}=g_{\phi_s}^\ast \om=:g_s^\ast\om$, and we define map
\begin{equation}
\nu:
\begin{array}{ccccc}
\overline\PSH(X,\om_X,J) & \xrightarrow{\hspace*{1.6cm}}  &\cJ_\Int \\
(\phi, \om_{\phi}) & \xmapsto{\hspace*{1.0cm}}  & (g_\phi^{-1})^\ast J&.
\end{array}
\end{equation}
where
$$
\PSH(X,\om_X,J):=\left\{\phi\ \Big| \ \om_\phi:=\om+\ddbar\phi>0, \int_X\phi\om^n=0\right\}
$$
and
$$
\overline\PSH(X,\om_X,J):=\left\{(\phi,g_\phi)\in \PSH(X,\om_X)\times \Diff(X)\ |\ g_\phi^\ast\om=\om_\phi\right\}\subset\PSH(X,\om_X,J)\times \Diff(X).
$$
Then we have that $\im\nu=\Ham^\CC(X,\om_X)/\Ham(X,\om_X)$. To see this one notice that, if we let $\red f:=g^{-1}$
\begin{eqnarray*}
&&\frac{d}{ds}\left(\om\right)=\frac{d}{ds}\left((g_s^{-1})^\ast\om_s\right)
=f_s^\ast L_{\dot f_s}\om_s+\ii f_s^\ast\dd_J\dbar_J\dot\phi_s=f_s^\ast d \om_s(\dot f_s,\bullet)+f_s^\ast dd_J\dot\phi_s
\end{eqnarray*}
with
$$
d_J:= J\circ d\text{ hence  } \dd=\frac{d+\ii d_J}{2} \text{ and } \dbar=\frac{d-\ii d_J}{2}
$$

\begin{eqnarray*}
0=\frac{d}{ds}\left(f^\ast_{s}\om_s\right)
&=&f_s^\ast L_{\dot f_s}\om_s+\ii f_s^\ast\dd_J\dbar_J\dot\phi_s=f_s^\ast d \big(\om_s(\dot f_s,\bullet)\big)+f_s^\ast dd_J\dot\phi_s\\
&=&d\Big(f_s^\ast  \big(\om_s(\dot f_s,\bullet)\big)+f_s^\ast d_J\dot\phi_s\Big)
=d\Big(f_s^\ast \big( \om_s(\dot f_s,\bullet)\big)+f_s^\ast \big(J\circ d\dot\phi_s\big)\Big)
\end{eqnarray*}
by solving $f_s^\ast\big(\om_s(\dot f_s,\bullet)\big)+f_s^\ast \big(d_J\dot\phi_s\big)=0$.
For $v\in T_xX$, we have
\begin{eqnarray*}
\om_s(\dot{f_s}|_{f_s(x)}, (f_s)_* v)\big|_{f_s(x)}
&=&\omega_s\Big(f_{s\ast}\circ (f_s^{-1})_*( \dot{f_s} |_{f_s(x)}), (f_s)_*(v|_x)\Big)\Big |_{f_s(x)}\\
\red\left(\because f_s^\ast\om_s=\om \right)
&=&\om\Big((f_s^{-1})_* \big(\dot{f_s}|_{f_s(x)}\big),(v|_x)\Big)\Big|_x
\end{eqnarray*}
and
\begin{eqnarray*}
f_s^*\big(J\circ (d\dot{\phi_s})|_{f_s(x)}\big)(v|_x)
&=&d\dot{\phi_s}|_{f_s(x)}\Big(J|_{f_s(x)} \big(f_{s\ast}v|_x\big)\Big)\\
&=&d\dot{\phi_s}|_{f_s(x)}\Bigg( f_{s\ast}\Big(\big(\stackbin[\red =J^{f_s}=J_s|_x]{}{\underbrace{ (f_s^{-1})_*J|_{f_s(x)} f_{s\ast}}}\big)v|_x\Big)\Bigg)\\
&=&d\dot{\phi_s}|_{f_s(x)}\Big ((f_{s\ast}\big( J_s|_x\cdot v|_x\big)\Big)\\
&=&\big( J_s\circ d (f_s^\ast\dot{\phi_s})\big)(v|_x)
\end{eqnarray*}
hence
$$
\om\Big((f_s^{-1})_* \big(\dot{f_s}|_{f_s(x)}\big),(\bullet |_x)\Big)\Big|_x=\big( J_s\circ d (f_s^\ast\dot{\phi_s})\big)(\bullet |_x)
$$
from which we obtain
$$
\fdiff(X)\ni \stackbin[\red g_{s\ast}\Big(\dot f_s\big |_{g_s^{-1}(x)}\Big)]{}{\underbrace{(f_s^{-1})_* \big(\dot{f}_s|_{f_s(x)}\big)\big|_x}}
=\big( J_s\circ d (f_s^\ast\dot{\phi_s})\big)\big|_x\in J_s\circ\fham(X,\om_X).
$$


\subsection{Legendre Transform}\label{Leg}

%

Let $\red F: \RR\supset U\rightarrow \RR$ be a convex function and $\bD_F: C^{\infty}(X,U)\rightarrow \RR$ be a functional defined by
\[\bD_F(\varphi):=\int_X F(\varphi(x)) d\vol_X\] 
for a fixed volume form $d\vol_X$. 

\begin{prop}\label{D-leg}
Let $P:=\{t\in \RR| t=f'(x) \text{ for some }x\in X\}=\im (dF)$
for any $\psi\in C^{\infty}(X, P)$. 
Then 
the Legendre transform $\widehat{\bD_F}: C^\infty(X,P)\to \RR$ is given by 

\[\widehat{\bD_F}(\psi)=\int_X \widehat{F}(\psi(x)) d\vol_X,\] 
with $\widehat{F}$ being the Legendre transform of $f$ corresponding to the pairing
\[\langle \varphi,\psi\rangle:=\int_X \varphi(x)\psi(x) d\vol_X.\]
\end{prop}
\begin{proof}
Fix $\psi$, then we have
\begin{align*}
\widehat{\bD_F}(\psi):=&\sup_{\varphi \in C^\infty(X,\RR)}\{\langle \varphi,\psi\rangle- \bD_F(\varphi)\}\\
=& \sup_{\varphi \in C^\infty(X,U)}\left(\int_X \varphi(x)\psi(x)d\vol_X-\int_X F(\varphi(x))d\vol_X\right)\\
=& \sup_{\varphi \in C^\infty(X,U)}\int_X \big(\varphi(x)\psi(x)-F(\varphi(x))\big)d\vol_X\\
\leq& \int_X \left(\sup_{u\in U} (u\cdot \psi(x)-F(u))\right) d\vol_X\\
=&\int_X \widehat{F}(\psi(x))d\vol_X.
\end{align*}

We also need to show that there is a $\varphi\in C^\infty(X,U)$ satisfying:
\[\int_X (\varphi(x)\psi(x)-F(\varphi(x)))d\vol_X\geq \int_X \widehat{F}(\psi(x))d\vol_X.\] 
For that, by our assumption $\psi\in C^\infty(X,P)$ for any $x\in \RR$ there is a {\red unique} $p_x$ satisfying $\psi(x)=F'(p_x).$   We define $\varphi_{\psi}(x)=p_x$, we only need to show that $\varphi_{\psi}(x)$ is continuous. However, $F$ is convex and smooth implies $(dF)^{-1}$ exists and smooth. Hence $\varphi_{\psi}(x)= (dF)^{-1}(\psi(x))$ which is  smooth, and for any $x$, $\widehat{F}(\psi(x))=\varphi(x)\psi(x)-F(\varphi(x))$, which implies the inequality is equality.
\end{proof}

\begin{exam}\label{F-check}
$F(u)=(u+1/V)\log (u+1/V)-(u+1/V)+1$ then $\widehat F(p)=e^p-p$. Then the functions 
$$
\begin{array}{rclcrclccl}
F'(u) &=&\log(u+1/V) &\text{ with }& \im(F')&=&(-\infty,+\infty),\ &  \Dom(F')&=&(-1/V,+\infty)\\
 \widehat F'(p)&=&e^p-1/V &\text{ with }& \im(\widehat F')&=& (-1/V,+\infty),\ &  \Dom(\widehat F')&=&(-\infty, +\infty)
\end{array}
 $$ 
are inverse to each other (cf. \eqref{exp-x}).  By Proposition \ref{D-leg}, one has
$$
\widehat{\bD_F}=\bD_{\widehat F}
$$
as we know  from \eqref{mu-J} that $u=\dfrac{\mu(J)}{\om^n/n!}-\dfrac{1}{V}>-\dfrac{1}{V}$, from which we deduce that
the corresponding $\mu$-invariant in Theorem \ref{Fut} is given by:
\begin{eqnarray*}
\la (d\bD_F)^{-1}(\theta_{\xi}),\theta_{\eta}\ra
=\la d\bD_{\widehat F}(\theta_{\xi}),\theta_{\eta}\ra
&=&\int_X\widehat F'(\theta_{\xi})\cdot\theta_{\eta}\frac{\om_X^n}{n!}=
\int_X\frac{e^{\hat \theta_\xi}-1}{V}\cdot\theta_{\eta}\frac{\om_X^n}{n!}\\
&=&\frac{1}{V}\int_X \left(\frac{e^{\theta_{\xi}}}{\displaystyle\frac{1}{V}\int_Xe^{\theta_{\xi}}\frac{\om^n_X}{n!}}-\frac{1}{V}\right)\theta_{\eta}\frac{\om_X^n}{n!}.
\end{eqnarray*}

\end{exam}


\end{document}